\documentclass[reqno,11pt]{amsart}
\usepackage{amsthm,amsfonts,amssymb,euscript,bbm,amsmath,mathrsfs}

\usepackage{empheq}
\usepackage{latexsym}
\usepackage{enumerate}
\usepackage{array}
\usepackage{cancel}
 \usepackage{color,ulem,textcomp}
\usepackage{hyperref}
\usepackage{pdfsync}

\newcommand{\wh}{\widehat}
\newcommand{\dis}{\displaystyle}
\newcommand{\bea}{\begin{eqnarray}}
\newcommand{\eea}{\end{eqnarray}}
\def\beaa{\begin{eqnarray*}}
\def\eeaa{\end{eqnarray*}}
\def\ba{\begin{array}}
\def\ea{\end{array}}
\def\be#1{\begin{equation} \label{#1}}
\def \eeq{\end{equation}}

\def\a{{\alpha}}

\def\be{{\beta}}

\def\eps{\epsilon}

\def\om{\omega}

\def\s{{\sigma}}

\def\AA{{\mathcal A}}

\def\NN{{\mathcal N}}

\def\HH{{\mathcal H}}

\def\TT{{\mathcal T}}
\def\WW{{\mathcal W}}

\def\UU{{\mathcal U}}

\def\HH{\mathcal H}

\def\DD{{\mathcal D}}

\def\RR{{\mathcal R}}

\def\AA{{\mathcal A}}

\def\R{{\mathbb{R}}}
\def\C{{\mathbb{C}}}

\def\N{{\mathbb N}_0}
\def\L{{\bf L}}
\def\O{{\bf O}}

\def\Z{{\mathbb{Z}}}

\def\T{{\mathbb{T}}}

\newcommand{\ov}{  \overline   }
\newcommand{\<}{  \langle   }
\renewcommand{\>}{  \rangle   }

\newcommand{\ligne}{\vspace{1\baselineskip}}

\newtheorem{theorem}{Theorem}[section]
\newtheorem{lemma}[theorem]{Lemma}
\newtheorem{proposition}[theorem]{Proposition}
\newtheorem{corollary}[theorem]{Corollary}

\newtheorem{remark}[theorem]{Remark}

\numberwithin{equation}{section}

\begin{document}

\author{Zaher Hani}
\address{School of Mathematics, Georgia Institute of Technology, Atlanta, GA 30332, USA}
\email{hani@math.gatech.edu}

\author{Laurent Thomann}
\address{Laboratoire de Math\'ematiques J. Leray, UMR  6629 du CNRS, Universit\'e de Nantes, 
2, rue de la Houssini\`ere,
44322 Nantes Cedex 03, France}
\email{laurent.thomann@univ-nantes.fr}

\thanks{Z.~H. is supported by a Simons Postdoctoral Fellowship, NSF Grant DMS-1301647, and a start-up fund from the Georgia Institute of Technology. L.T. is  supported   by the  grant  ``ANA\'E'' ANR-13-BS01-0010-03.}

\title[NLS with harmonic potential]{Asymptotic behavior of the nonlinear Schr\"odinger equation with harmonic trapping}


\begin{abstract}
We consider the cubic nonlinear Schr\"odinger equation with harmonic trapping on\;$\R^D$ ($1\leq D\leq 5$). In the case when all but one directions are trapped (a.k.a ``cigar-shaped" trap), following the approach of~\cite{HPTV}, we prove modified scattering and construct modified wave operators for small initial and final data respectively. The asymptotic behavior turns out to be a rather vigorous departure from linear scattering and is dictated by the resonant system of the NLS equation with full trapping on $\R^{D-1}$. In the physical dimension $D=3$, this system turns out to be exactly the (CR) equation derived and studied in~\cite{FGH, GHT1, GHT2}. The special dynamics of the latter equation, combined with the above modified scattering results, allow to justify and extend some physical approximations in the theory of Bose-Einstein condensates in cigar-shaped traps. 
\end{abstract}
\maketitle


\maketitle

\section{Introduction}

The aim of this paper is to study the long-time behavior of the cubic nonlinear Schr\"odinger equation with harmonic trapping given by 
\begin{equation}\label{GeneralNLS}
(i\partial_t  -\Delta_{\R^D} + \sum_{j=1}^D\omega_j x_j^2) U=\kappa_0 |U|^2U,\qquad (x_1, \ldots, x_d)\in \R^D,
\end{equation}
with a particular emphasis on the anisotropic limit $\omega_1=0< \omega_2=\ldots=\omega_D$. Here $\omega_j$ signifies the frequency of the harmonic trapping in the $j-$th direction and $\kappa_0 \neq 0$.\medskip

The motivation for this study is two-fold: On the one hand, we aim at justifying some approximations done in the physics literature that allow reducing the dynamics of~\eqref{GeneralNLS} in the highly anisotropic setting (a.k.a. cigar-shaped trap) to that of the homogeneous (i.e. with no trapping) 1D cubic NLS equation. Such approximations, often referred to as the ``quasi-1D dynamics"~\cite{KFC}, allow access to the complete integrability theory of the 1D cubic NLS equation along with its plethora of special solutions that give theoretical explanations of fundamental phenomena in Bose-Einstein condensates. On the other hand, from a purely mathematical point of view, the analysis falls under the recent progress and interest in understanding the asymptotic behavior of nonlinear dispersive equations in the presence of a confinement. Such a confinement can come from the compactness (or partial compactness) of the domain or via a trapping potential. In either case, this leads to the complete or partial loss of dispersive decay of linear solutions, and consequently complicating and diversifying the picture of long-time dynamics. In this line, using tools developed for the study of long-time dynamics of nonlinear Schr\"odinger equations on product spaces, we will be able to describe the asymptotic dynamics and show that they exhibit highly nonlinear behavior in striking contrast to linear scattering. As a consequence of this description, we get the general extension of the ``quasi-1D approximation" mentioned above to cases when higher and multiple energy levels of the harmonic trap are excited (cf. Section~\ref{sect13}). \medskip 
 
\subsection{Physical context} A Bose-Einstein condensate (BEC for short) is an  aggregate of matter (Bosons) which appears at very low temperature and which is due to the fact that all particles are in the same quantum state.  Their existence was predicted by Bose in 1924 for photons and by Einstein in 1925 for atoms, and they were experimentally observed in 1995 by   W.~Ketterle,   A.~Cornell  and C.~Wieman who were awarded a Nobel Prize shortly after, in 2001, for this achievement. This observation was followed by a burst of activity in the theoretical and experimental study of BEC which constitutes a rare manifestation of a quantum phenomenon which shows through at a macroscopic level. For an nice introduction to this topic we refer to the book~\cite{KFC} and to~\cite{Lew}.\medskip 


 In the physical space $\R^3$, BEC can be realized by trapping particles using a magnetic trap which is modelled in the mean-field theory by the harmonic potential term in~\eqref{GeneralNLS}. The wave function $U(t,x,y_1,y_2)$ of the particles in~\eqref{GeneralNLS} (with $D=3$) 
can be interpreted as the probability density of finding particles at point $(x,y_1, y_2) \in \R^3$ and time $t\in \R$. The sign $\kappa_0=+1$ or $-1$ depends on whether the Boson interaction is attractive (focusing case) or repulsive (defocusing). \medskip

In the case when $\omega_1 \ll \omega_2=\omega_3$, the harmonic trap is often described as ``cigar-shaped", and we will be interested in this case. This regime is of great importance from the physical point of view as it allows for a ``dimensional reduction" in which the condensate is described by better-understood lower-dimensional dynamics. More precisely, a naturally adopted approximation of\;\eqref{GeneralNLS} is obtained by going to the anisotropic limit and setting $\om_1=0$ (which is justified for $x$ not too large) and non-dimensionalizing $\omega_2=\omega_3=1$. Then, the resulting equation is
\begin{equation}\label{PNLS}
(i\partial_t  -\Delta_{\R^3} +   y_1^2+  y_2^2) U=\kappa_0 |U|^2U,\qquad (x,y_1, y_2) \in \R^3.
\end{equation}

\medskip
 In this context, (see for instance~\cite{PGHH} or~\cite[paragraph\;1.3.2]{KFC}) physicists often adopt an Ansatz of the form
 \begin{equation}\label{approx}
 U(t, x, y)\sim \psi(t, x) e^{2i t} e^{-\vert y\vert^2/2}; \qquad y=(y_1, y_2),
 \end{equation}
which leads them through a multiple time-scale expansion to the 1D-dynamics obeyed by $\psi(t,x)$. This dynamics is given by none other than the one dimensional Schr\"odinger equation
\begin{equation} \label{1dNLS} 
  \left\{
      \begin{aligned}
      &(i\partial_t -\partial^2_x)\psi=\kappa_0 \lambda_0 |\psi|^2\psi, \quad (t,x)\in \R\times \R,
       \\  &  \psi(0,x)  =\varphi(x).
      \end{aligned}
    \right.
\end{equation}
This equation is obtained by projecting the nonlinear term in the Ansatz equation on the groundstate $g_{0}(y)=e^{-\vert y\vert^2/2}$ of the harmonic oscillator $-\Delta_{\R^2}+\vert y\vert^2$, thus $\lambda_0=\Vert  g_{0}\Vert^4_{L^4(\R^2)}/\Vert  g_{0}\Vert^2_{L^2(\R^2)}$=1/2. \medskip

The aim of this paper  is to describe the large time behaviour of~\eqref{PNLS} for small, smooth and decaying data (and more generally for the corresponding equations in higher   dimension). One consequence of our work is a justification of the approximation~\eqref{approx} for large times, as well as the correct extension of that approximation when higher and/or multiple energy levels of the quantum harmonic oscillator are excited. We give the relevant result concerning the approximation~\eqref{approx} and refer to the next section for  more general and precise results. Denote by $\mathcal{S}(\R)$ the set of the Schwartz functions, then 
\begin{theorem}\label{thm0}
Let $\varphi\in \mathcal{S}(\R)$ be small enough, and let $\psi$ be the solution of~\eqref{1dNLS}. Then  there exists a solution $U\in \mathcal{C}\big([0,+\infty); L^{2}(\R\times \R^{2})\big)$ of\;\eqref{PNLS} such that 
\begin{equation*}
\big\Vert U(t,x,y)-\psi(t,x)  e^{2it} e^{-\frac{1}{2}|y|^{2}}\big\Vert_{L^2(\R\times \R^2)}\longrightarrow 0\quad\hbox{ as }\;\;\;t\longrightarrow+\infty.
\end{equation*}
Moreover, the function $U$ can be chosen to be axisymmetric: $U(t,x,y)=\widetilde{U}(t,x, |y|)$ for some  $\widetilde{U}$.
\end{theorem}
 This shows that the 1D dynamics of~\eqref{1dNLS} can be embedded in the 3D dynamics of~\eqref{PNLS}, a reduction, known as {\it quasi 1D dynamics}, which is at the basis of the theoretical explanation of many fundamental phenomena in Bose-Einstein condensates. Physicists arrive at it using some multiple time-scale approximations, and use it afterwards to transfer information from the well-understood and completely integrable dynamics of~\eqref{1dNLS} to that of~\eqref{CNLS}. We refer to Section\;\ref{sect13} for more details and extensions.

We remark that the nonlinear Schr\"odinger equation with harmonic potential is also an important model in nonlinear optics, but we do not detail the applications in this context. 

\subsection{Mathematical framework and results} From the mathematical perspective, this work falls under the recent developments in the study of the long-time behavior of nonlinear dispersive equations that are confined either by physical domain or by a potential. Such a confinement has a direct effect on the linear decay afforded by dispersion which either completely disappears or becomes very weak. In the case of a confinement by domain, we refer to~\cite{HPTV} for references on global regularity issues. The question of long-time behavior is much more subtle and much less is known especially on compact domains. In that case, many different long-time dynamics can be sustained starting from arbitrarily small neighborhoods of the zero initial data including KAM tori~\cite{BourgainQuasi, EliaKuk, KukPo, Procesi}, Arnold diffusion orbits~\cite{CKSTTTor, GuKa}, and coherent frequency dynamics~\cite{FGH}. 

A bit more can be said about the asymptotic behavior in the partially periodic cases ({\it e.g.} waveguides $\R^d\times \T^n$) where one can see an interplay between the persistence of the Euclidean behavior of scattering~\cite{TzVi} versus the emergence of new asymptotic dynamics that reflect the (global) non-Euclidean geometry of the underlying manifold. The latter was observed in~\cite{HaPa, HPTV} that studied NLS on the product domain $\R\times \T^d$ which corresponds to the non-compact quotient of $\R^{d+1}$ where linear solutions decay the least, namely like $t^{-1/2}$. There, the resonant periodic interactions played a key role in the asymptotic dynamics either for all small data~\cite{HPTV} or in the analysis of certain sequences of data that appear in the profile decomposition~\cite{HaPa}. 

We will see that a similar picture can be drawn in the case of a potential trap confinement. Studying (NLS) under full harmonic trapping would be the analogue of studying it on a compact domain, whereas studying it with partial trapping corresponds to a partially compact domain. Indeed, our analysis will draw a lot on the recent advances in the latter study, especially on the recent work~\cite{HPTV}. \medskip

We now present our general results. Let $k,d\geq 1$. We define the Laplace operator  $\Delta_{\R^k}=\sum_{j=1}^k \partial^2_{x_k}$ on $\R^k$ and the harmonic oscillator  $\HH_d=\sum_{j=1}^d \big(-\partial^2_{y_j}+\vert y_j\vert^2\big) $ on $\R^d$. Our equation of interest can now be written as 
 \begin{equation*} 
  \left\{
      \begin{aligned}
      &(i\partial_t -\Delta_{\R^k}+\HH_d) U=\kappa_0\vert U\vert^2U, \quad (t,x,y)\in \R\times \R^{k}\times \R^{d},
       \\  &  U(0,x,y)  =U_0(x,y),
      \end{aligned}
    \right.
\end{equation*}
where $U$ is a complex-valued function, $\kappa_0 \in \{+1, -1\}$ is the sign of the nonlinearity ($+1$ for focusing and $-1$ for defocusing). \medskip

The case $k=0$ corresponds to full harmonic trapping. Since the spectrum of $\HH_d$ is discrete, the linear solution does not decay in this case. As was the case for a geometric confinement, this has little effect on global regularity issues~\cite{KV, Jao}. As a result, the asymptotic behavior of~\eqref{CNLS} can be quite rich, but very little is known, apart from the existence of many periodic solutions\;\cite{Carles2} and the existence of KAM tori
when $d=1$, if one allows a perturbation by a multiplicative potential\;\cite{GrTh1}.\medskip

 For $k\geq 1$, the linear solution decays typically like $t^{-k/2}$. This leads to scattering when $k\geq 2$, as was shown by Antonelli, Carles and Silva~\cite{AnCaSi}. We also refer to~\cite{AnCaSi} for a review of scattering theory for NLS. \ligne

 The case $k=1$, in which all but one directions are trapped, is particularly interesting since it corresponds to a long-range nonlinearity\footnote{If linear solutions decay like $t^{-k/2}$, the nonlinearity $|u|^{p-1}u$ is short-range if $p>1+\frac{2}{k}$ and long-range otherwise.}. We will be interested in this case in which the equation can be written as
 \begin{equation} \label{CNLS} 
  \left\{
      \begin{aligned}
      &(i\partial_t + \DD) U=\kappa_0\vert U\vert^2U, \quad (t,x,y)\in \R\times \R\times \R^{d},
       \\  &  U(0,x,y)  =U_0(x,y),
      \end{aligned}
    \right.
\end{equation}
where we have set
\begin{equation*}
\DD=-\partial^2_x+\mathcal{H}_{d}.
\end{equation*}
Since the nonlinearity decays at best like $t^{-1}$ in $L^2$, a modification of the scattering operator might be needed to describe the asymptotic dynamics. This is indeed the case and the modification will be given by the {\it resonant system} of this equation. Actually, such a phenomenon has been put to evidence by the first author, Pausader, Tzvetkov and Visciglia\;\cite{HPTV} for the cubic nonlinear Schr\"odinger on $\R\times \T^{d}$. Here we follow their general strategy, which is also described in~\cite{HPTVproc}.

In order to define the limit system, we start by recalling that if $v$ solves
\begin{equation*} 
(i\partial_t +\HH_d)v=\kappa_0 |v|^2v\qquad (t, y) \in \R\times \R^d
\end{equation*}
and if we look at the {\it profile} $f(t)=e^{-it\HH_d} v$, then it satisfies the equation 
$$
i\partial_t f=\kappa_0 e^{-it\HH_d}(|e^{it\HH_d} f|^2 e^{it\HH_d} f)=\kappa_0\sum_{n_{1},n_{2}, n_3, n \in \N}e^{it\omega}\Pi_n\left(\Pi_{n_{1}}f \overline{\Pi_{n_{2}}f} \Pi_{n_3}f\right)
$$
where $\N$ is the set of non-negative integers, $\Pi_n$ is the projection onto the  $n-$th eigenspace of $\HH_d$, and $\omega=2(n_{1}-n_{2}+n_3-n)$. The resonant system associated to the previous equation is obtained by restricting the above sum to the {\it resonant interactions} corresponding to $\omega=0$. This is given explicitly by 
\begin{equation}\label{RS}
\begin{split}
i\partial_t g(t,y)=&\kappa_0 \mathcal T[g, g, g](y)\\
\mathcal T[f,g,h]:=&\sum_{\substack{n_{1},n_{2}, n_3, n \in \N\\ n_{1}+n_3=n_{2}+n}}  \Pi_n\left(\Pi_{n_{1}}f \overline{\Pi_{n_{2}}g} \Pi_{n_3}h\right)=\sum_{n_{1},n_{2}, n_3 \in \N     }\Pi_{n_{1}+n_3-n_{2}}\left(f_{n_{1}}\ov{g_{n_{2}}}h_{n_3}\right),
\end{split}
\end{equation}
where we will often denote $w_n(y)=\Pi_n w(y)$ for an arbitrary function $w$ on $\R^d$. When $d=2$, we will see that this system is none other than the (CR) equation derived in~\cite{FGH} (see Section~\ref{sect33} for more details). \medskip

We are now ready to describe the asymptotic dynamics of~\eqref{CNLS}. This will be given by what we shall call the {\it limit system} defined by
 \begin{equation} \label{RSS} 
 \begin{split}
&  \left\{
      \begin{aligned}
      &i\partial_\tau G(\tau)=\kappa_0\mathcal R[G(\tau), G(\tau), G(\tau)], \quad (\tau,x,y)\in \R\times \R\times \R^{d},
       \\  &  G(0,x,y)  =G_0(x,y),
      \end{aligned}
    \right.\\[3pt]
  &  \text{where}\qquad\mathcal F_{x}\mathcal R[G, G, G](\xi,\cdot)=\mathcal T [\widehat G(\xi, \cdot),\widehat G(\xi, \cdot),\widehat G(\xi, \cdot)].
    \end{split}
\end{equation}
Here, we denoted $\widehat G(\xi,y)$ the partial Fourier transform in the first (un-trapped) variable. Noting that the dependence on $\xi$ is merely parametric, the above system is essentially~\eqref{RS}.\medskip

In the sequel, for $N\geq 1$, $H^{N}(\R\times \R^{d})$ is a Sobolev space and $S=S_N$, $S^+=S_N^+$ denote Banach spaces whose norms are defined in~\eqref{DefSNorm} later. They are contained in $H^N$ and include all the Schwartz functions.

\begin{theorem}\label{ModScatThm}
Let $1\leq d \leq 4$ and $N\geq8$. There exists $\varepsilon=\varepsilon(N,d)>0$ such that if $U_0\in S^+$ satisfies
\begin{equation*}
\Vert U_0\Vert_{S^+}\le\varepsilon,
\end{equation*}
and if $U(t)$ solves~\eqref{CNLS} with initial data $U_0$, then $U\in \mathcal{C}([0,+\infty); H^N(\R\times \R^d))$ exists globally and exhibits modified scattering to its resonant dynamics~\eqref{RSS} in the following sense: there exists $G_0\in S$ such that if $G(t)$ is the solution of~\eqref{RSS} with initial data $G(0)=G_0$, then
\begin{equation*}
\Vert U(t)-e^{it\DD}G(\pi\ln t)\Vert_{H^N(\R\times \R^d)}\longrightarrow 0\quad\hbox{ as }\;\;\;t\longrightarrow+\infty.
\end{equation*}
\end{theorem}
By reversibility of the equation, a similar statement holds as $t\longrightarrow-\infty$.  \ligne

Our next result is the existence of  modified wave operators.

\begin{theorem}\label{ExMWO}
Let $1\leq d\leq 4$ and $N\geq8$. There exists $\varepsilon=\varepsilon(N,d)>0$ such that if $G_0\in S^+$ satisfies
\begin{equation*}
\Vert G_0\Vert_{S^+}\le\varepsilon,
\end{equation*}
and $G(t)$ solves~\eqref{RSS} with initial data $G_0$, then there exists $U\in \mathcal{C}\big([0,\infty); H^N(\R\times \R^d)\big)$ which is a solution of\;\eqref{CNLS} such that
\begin{equation*}
\begin{split}
\Vert U(t)-e^{it\DD}G(\pi\ln t)\Vert_{H^N(\R\times \R^d)}\longrightarrow0\;\;\hbox{ as }\;\;t\longrightarrow+\infty.
\end{split}
\end{equation*}
Moreover 
\begin{equation*}
\Vert U(t)\Vert_{L_x^{\infty}\HH^1_y}\leq C(1+\vert t\vert)^{-1/2}.
\end{equation*}
\end{theorem}

As a result, any dynamics of the limit system~\eqref{RSS} with data in $S^{+}$, gives information on the  asymptotic behaviour of~\eqref{CNLS}. We point out that the condition $N\geq 8$ is not optimal in the previous results.

\subsection{Physical interest and particular dynamics in dimension \texorpdfstring{$\boldsymbol {D=3}$}{D=3}}\label{sect13}  We are now able to precise the results of the first paragraph. The key point is that when $D=3$, the corresponding resonant system~\eqref{RS} is the so-called (CR) equation derived and studied in~\cite{FGH, GHT1, GHT2}. This (CR) system-standing for {\it Continuous Resonant} (cf. equation~\eqref{star})- was first obtained as the large box (or continuous) limit of the resonant system for the {\it homogeneous} cubic NLS equation on a box of size $L$, and was shown to approximate the dynamics of the cubic NLS equation over very long nonlinear time scales. It can be also understood as the equation of the high-frequency envelopes of NLS on the unit torus $\T^2$. The rather surprising fact that it is also the resonant system~\eqref{RS} of NLS with harmonic trapping is one consequence of its rich properties and behavior (cf. Section~\ref{sect33}). As a consequence, one can use Theorem~\ref{ExMWO} to construct other interesting non-scattering dynamics for equation~\eqref{CNLS} as is illustrated in the following result.

\begin{corollary} \label{coro}
Let $d=2$. We let $F(t)=e^{-it\DD} U(t)$ denote the profile of $U$, then   
\begin{enumerate}
\item No nontrivial scattering: Assume that $U$ solves~\eqref{CNLS} and that $F(t)$ converges as ${t\to +\infty}$. If $U(0)$ is sufficiently small, then $U\equiv 0$.
\item Quasi-periodic frequency dynamics:
There exists a global solution $U(t)$ such that $\widehat F(e^{t},\xi)$ is asymptotically quasi-periodic with two distinct frequencies for all $\xi \in [-1,1]$.

\end{enumerate}
\end{corollary}
We refer to Section~\ref{sect33} for more details and basic properties of the (CR) equation. An interesting open question is to decide whether growth of Sobolev norms can occur in the (CR) equation as was the case for the resonant system on the torus $\T^d$ observed in~\cite{Hani}. This would in turn imply the same behavior for~\eqref{CNLS} (along with a proof of arbitrary large but finite growth of Sobolev norms for the equation with full trapping\footnote{Indeed,~\cite[Theorem 2.6]{FGH} also implies that such a result would also give another proof of the growth result for the cubic NLS equation on $\T^2$ first obtained in~\cite{CKSTTTor}.}). \medskip

 Let us now justify the approximation~\eqref{approx}.   For a solution $U$ of~\eqref{CNLS}, define $V=e^{-it\HH_d}U$ and write the decomposition $V=\sum_{p\in \N}V_p$ according to the eigenspaces of $\HH_d$. Then the equations satisfied by $V_p(t,x,y)$ (with $y=(y_1,y_2)$) are given by
$$
(i\partial_t V_p-\partial^2_{x} V_p)= \kappa_0\mathcal{T} (V, V, V)_p +\mathcal E_p(t)=\kappa_0\sum_{p_{1}-p_{2}+p_3=p}\Pi_p \left(V_{p_{1}}\overline{V_{p_{2}}}V_{p_3}\right)+\mathcal E_p(t),
$$
and we show in the proof that $\mathcal E_p(t)$ is integrable in time. So at a heuristic level, the approximation consists of dropping $\mathcal E$ so that the large-time effective equation for $V(t,x,y)$ is thus given by
\begin{equation}\label{model}
(i\partial_t -\partial^2_{x})V= \kappa_0\mathcal{T} \left(V(t, x, \cdot), V(t,x, \cdot), V(t,x, \cdot)\right).
\end{equation}
Recall that $g_{n}(y):=(y_{1}+i y_{2})^ne^{-\frac{\vert y \vert^2}{2}}$ is an eigenfunction of $\HH_{2}$ and assume that the initial data at $t=0$ is given by $V(0,x,y)=\psi(x) (y_{1}+i y_{2})^ne^{-\frac{\vert y \vert^2}{2}}$ (this means that $V_p(0)=0$ for $p\neq n$). In the language of Bose-Einstein condensates such data correspond to {\it vortices} of degree $n$ (see~\cite{KFC}).   Then by~\cite{GHT1}, it holds that
\begin{equation}\label{gn}
\mathcal{T}\big(g_{n},g_{n},g_{n}\big)=\lambda_n g_{n},
\end{equation}
with\footnote{The constant $\lambda_{n}$ is merely computed by taking the scalar product with $g_{n}$ in the definition~\eqref{RS}.} 
$$\lambda_n:=\frac{\Vert  g_{n}\Vert^4_{L^4(\R^2)}}{\Vert  g_{n}\Vert^2_{L^2(\R^2)}}=\frac{(2n)!}{2^{2n+1}n!}.$$
 Hence we get that the solution of~\eqref{model} stays in the form 
\begin{equation}\label{quasi}
V(t,x,y)=\psi(t,x) (y_{1}+i y_{2})^ne^{-\frac{\vert y \vert^2}{2}},
\end{equation}
where $\psi$ satisfies the 1D NLS equation
\begin{equation}\label{1dNLS1}
(i\partial_t -\partial^2_{x}) \psi(t,x)=\kappa_0\lambda_n |\psi(t,x)|^2\psi(t,x).
\end{equation}

Two applications of this ``quasi-1D dynamics" in the physics literature is to give theoretical explanations of the all-important {\it bright and dark solitons} observed in cigar-shaped harmonic traps. Bright solitons correspond to the usual soliton solutions of~\eqref{1dNLS1} in the focusing case ($\kappa_0=1$). More importantly, dark solitons correspond to the so-called dark soliton solutions of\;\eqref{1dNLS1}: these are stationary solutions of the defocusing problem (obtained from the complete integrability theory) that exhibit an ``amplitude dip" form a constant amplitude of 1 at $\pm \infty$ to~0 near the origin. For more details, we refer to~\cite{Frantz}.

\medskip 

The decay in time of solutions is very important for our proof to work. This means that we cannot handle 1D- soliton solutions (for the focusing case) or dark-soliton data (which does not even decay at spatial infinity.) But for sufficiently small decaying data, one can go one step further than~\eqref{1dNLS1} in the asymptotic analysis of\;\eqref{CNLS}  (see~\cite{HaNa, Ca, LinSof, KaPu, IfTa} for related works on NLS): the effective large-time dynamics of $\psi(t,x)$ is actually given by
 \begin{equation*}  
(i\partial_t -\partial^2_{x}) \psi(t,x)=  \frac{\pi \kappa_0\lambda_n}{t}   \mathcal{F}_{\xi \to x}^{-1}\big(\vert \wh\psi(t,\xi) \vert^2\wh\psi(t,\xi)\big), \quad (t,x)\in \R\times \R,
\end{equation*}
which is solved by 
\begin{equation}\label{psi1}
\psi(t,x) = \mathcal{F}_{\xi\to x}^{-1}\Big( \wh{\varphi}(\xi) e^{i  \pi \kappa_{0} \lambda_n  \vert \wh\varphi(\xi)\vert^2 \ln t} e^{i(1-t)\xi^2}\Big) ,
\end{equation}
where $\varphi(x)=\psi(1,x)$. This will come as a consequence of our main Theorems~\ref{ModScatThm} and~\ref{ExMWO}.

Let us sum up in the  next proposition some asymptotic results of~\eqref{CNLS}, concerning quasi-1D dynamics. The case $n=0$ implies the result of Theorem~\ref{thm0}.

\begin{proposition}\label{prop15} Let $N\geq 8$. Let $\varphi\in \mathcal{S}(\R)$ and set $G_{0}(x,y)=\varphi(x) (y_{1}+i y_{2})^ne^{-\frac{\vert y \vert^2}{2}}$.
\begin{enumerate}
\item Consider the solution $\psi$ of~\eqref{1dNLS1}. Then if   $\|G_{0}\|_{S^{+}}\leq \eps$ is small enough, there exists a solution $U\in \mathcal{C}\big([0,\infty); H^N(\R\times \R^2)\big)$ of\;\eqref{CNLS} such that 
\begin{equation*}
\big\Vert U(t,x,y)-\psi(t,x)  e^{2i(n+1)t} (y_{1}+i y_{2})^ne^{-\frac{\vert y \vert^2}{2}}\big\Vert_{H^N(\R\times \R^2)}\longrightarrow 0\quad\hbox{ as }\;\;\;t\longrightarrow+\infty.
\end{equation*}
\item Consider the solution $\psi$ of~\eqref{psi1}. Then if   $\|G_{0}\|_{S^{+}}\leq \eps$ is small enough, there exists a solution $U\in \mathcal{C}\big([0,\infty); H^N(\R\times \R^2)\big)$ of\;\eqref{CNLS} such that 
\begin{equation*}
\big\Vert U(t,x,y)-\psi(t,x)  e^{2i(n+1)t} (y_{1}+i y_{2})^ne^{-\frac{\vert y \vert^2}{2}}\big\Vert_{H^N(\R\times \R^2)}\longrightarrow 0\quad\hbox{ as }\;\;\;t\longrightarrow+\infty.
\end{equation*}
\end{enumerate}
Moreover, in both of the previous cases, $U$ satisfies for all $\theta\in \R$, $U(t,x,R_\theta y)=e^{in\theta} U(t,x,y)$, where we have set $R_{\theta}y=e^{i\theta}y$ with the identification $y\equiv y_1+iy_2$.
\end{proposition}

There also exist solutions of~\eqref{CNLS} which decompose into a sum of two functions of the previous type. This will be the case of the {\it vortex dipole} in which the solution is a superposition of a positive vortex $(y_{1}+i y_{2})e^{-\frac{\vert y \vert^2}{2}}$ and its antipode $(y_{1}-i y_{2})e^{-\frac{\vert y \vert^2}{2}}$. In this case, the dynamics also reduces to a quasi-1D one, except that it is now given by the system 
 \begin{equation} \label{syst1}
  \left\{
      \begin{aligned}
      &(i\partial_t -\partial^2_{x}) \psi_{+}(t,x)=\frac{\kappa_{0}}{4}\big( 2|\psi_{-}|^{2}+|\psi_{+}|^{2}        \big) \psi_{+}(t,x), \quad (t,x)\in \R\times \R,\\
       &(i\partial_t -\partial^2_{x}) \psi_{-}(t,x)=\frac{\kappa_{0}}{4}\big( 2|\psi_{+}|^{2}+|\psi_{-}|^{2}        \big) \psi_{-}(t,x) ,  \\
       & \psi_{+}(0,x)=\varphi_{+}(x),\quad  \psi_{-}(0,x)=\varphi_{-}(x).
      \end{aligned}
    \right.
\end{equation}
The above coupled equation is known as the XPM system and it is a useful model in nonlinear optics\footnote{We thank Panos Kevrekidis for pointing this out.}  (see for example~\cite[Chapter 9]{CSS}). We do not know if it appeared in the study of Bose-Einstein condensates prior to this work. In what follows, denote $G_{\pm}(x,y)=\varphi_{\pm}(x) (y_{1}\pm i y_{2})e^{-\frac{\vert y \vert^2}{2}}$. 


\begin{proposition}\label{prop16}
  Let $\|G_{+}\|_{S^{+}},  \|G_{-}\|_{S^{+}}\leq \eps$ and consider the solutions $\psi_{+},\psi_{-}$ of~\eqref{syst1}. Then if $\eps>0$ is small enough, there exists a solution $U\in \mathcal{C}\big([0,\infty); H^N(\R\times \R^2)\big)$ of\;\eqref{CNLS} such that 
\begin{equation*}
\big\Vert U(t,x,y)-\psi_-(t,x) e^{4it}   (y_{1}-i y_{2})e^{-\frac{\vert y \vert^2}{2}}-\psi_+(t,x) e^{4it} (y_{1}+i y_{2})e^{-\frac{\vert y \vert^2}{2}}\big\Vert_{H^N(\R\times \R^2)}\longrightarrow 0\quad\hbox{as}\;\;\;t\longrightarrow+\infty.
\end{equation*}
\end{proposition}
 
 There is also a statement similar to the point (2) of Proposition~\ref{prop15}.\medskip

A natural question is whether every solution $U$ of~\eqref{CNLS}, with initial condition $\Vert U_0\Vert_{S^+}\leq \eps$,  asymptotically decouple into a finite sum of the type $\dis \sum_{\it{finite}} \psi_j(t,x)f_j(t,y)$ when $t\longrightarrow +\infty$. The next result provides a negative answer. Denote by $ \mathcal{C}_{b}\big([0,+\infty)\big)$ the set of continuous and bounded functions on $\R_{+}$, then 

\begin{proposition}\label{prop17}
For all $\eps>0$, there exists $\Vert U_0\Vert_{S^+}\leq \eps$ and a solution $U\in \mathcal{C}\big([0,+\infty); H^N(\R\times \R^2)\big)$ to~\eqref{CNLS} so that  
\begin{equation*}
\big\Vert  U(t,x,y)-    \sum_{\text{finite}} \psi_j(t,x)f_j(t,y)   \big \Vert_{L^2(\R\times \R^2)} \not \longrightarrow 0, 
\end{equation*}
when $t\longrightarrow +\infty$, for any  $\psi_j \in \mathcal{C}_{b}\big([0,+\infty); L^{2}(\R)\big)$ and $f_j\in \mathcal{C}_{b}\big([0,+\infty); L^{2}(\R^{2})\big)$.
\end{proposition}

\subsection{Ideas of the proof and comments} Let us explain how we can formally derive the {\it limit equation}. We expand functions $F=\sum_{q\geq 0}F_q$ according to the eigenspaces of the harmonic oscillator ($F_q=\Pi_q F$), and let $F(t,x,y)=e^{-it\DD}U(t,x,y)$. Defining the Fourier transform on $\mathbb{R}$ by
\begin{equation*}
\mathcal{F}_{x}g(\xi)=\widehat{g}(\xi):=\frac{1}{2\pi}\int_{\mathbb{R}}e^{-ix\xi}g(x)dx,
\end{equation*}
we will establish in Section\;\ref{sect22} that $U$ solves~\eqref{CNLS} if and only if $F$ solves
\begin{equation} \label{FeqnIntro}
i\partial_t F(t) = \mathcal{N}^{t}[F(t),F(t),F(t)],
\end{equation}
where the nonlinear term $\mathcal{N}^{t}$ is defined by 
\begin{multline}\label{Nintro}
\mathcal{F}_x\mathcal{N}^{t}[F,G,H](\xi,y)=\\
=\kappa_0\sum_{\omega\in 2 \Z}e^{-it\omega}\sum_{(p,q,r,s)\in \Gamma_{\omega}} \Pi_{p}\int_{\mathbb{R}^2}e^{2it\eta\kappa} \widehat{  F}_{q}(\xi-\eta,y)\overline{\widehat{ G}_{r}}(\xi-\eta-\kappa,y)\widehat{  H}_{s}(\xi-\kappa,y)d\eta d\kappa .
\end{multline}
The idea is to write the nonlinearity $\mathcal N$ in~\eqref{FeqnIntro} as a sum of an effective part $\mathcal N_{eff}$ and an integrable part $\mathcal E$. The integrable part does not affect the asymptotic dynamics of $F$, which should therefore be described by the equation $i\partial_t F=\mathcal N_{eff}(F)$. This effective or {\it limit} equation turns out to be exactly~\eqref{RSS}, a fact that can be seen by performing two heuristic approximations:\medskip 

{\it The normal form reduction:} It is a general principle that the main contribution in the nonlinearity is given by the resonant terms, namely corresponding to $\om=0$. Roughly speaking, the other terms can -in principle- be shown to be perturbative thanks to integration by parts in time and using the decay in time of $\partial_t F$. Equivalently, this consists in applying a Birkhoff normal form. However, in this case there is a problem if the time derivative hits the other phase factor\;$e^{2it\eta\kappa}$ in the regime when $\eta \kappa$ is not small. In this case, one has to rely on the spatial oscillations of the integral in~\eqref{Nintro} (including resorting to refined 1D bilinear Strichartz estimates). The upshot is that the remaining term after this reduction reads
\begin{equation*}
\mathcal{F}_x\mathcal{N}_{nf}^{t}[F,G,H](\xi,y)
:=\kappa_0\sum_{(p,q,r,s)\in \Gamma_{0}} \Pi_{p}\int_{\mathbb{R}^2}e^{2it\eta\kappa} \widehat{  F}_{q}(\xi-\eta,y)\overline{\widehat{ G}_{r}}(\xi-\eta-\kappa,y)\widehat{  H}_{s}(\xi-\kappa,y)d\eta d\kappa .
\end{equation*}

{\it Stationary phase asymptotics:} The asymptotics of the above term when $t\to +\infty$ can be obtained thanks to a stationary phase argument in $(\eta,\kappa)$, which suggests that 
\begin{equation*}
\mathcal{F}_x\mathcal{N}_{nf}^{t}[F,G,H](\xi,y)
\sim  \frac{\pi \kappa_0}t\sum_{(p,q,r,s)\in\Gamma_0}\Pi_p\Big(\widehat{F}_q(\xi,y)\overline{\widehat{G}_r}(\xi,y)\widehat{H}_s(\xi,y)\Big):=\frac{\pi \kappa_0}t\mathcal{F}_x\mathcal{R}[F,G,H](\xi,y).
\end{equation*}
This is precisely the nonlinearity in~\eqref{RSS}, after the change of time variable $\tau =\pi \ln t$. 

The point of the proof consists of making rigorous the above heuristics, and one of the key steps is to establish Proposition~\ref{StrucNon} which proves that the error committed in the above approximation is integrable in time. As was the case in~\cite{HPTV}, a big difficulty comes from the fact that we cannot establish the decay $\|U\|_{L^\infty}\lesssim t^{-1/2}$, which requires us to use a hierarchy of three norms to control our solutions $Z\subset S \subset S^+$. The $Z$ norm is conserved for the limit equation~\eqref{RSS} and will be uniformly bounded for~\eqref{CNLS}. The $S$ and $S^+$ grow slowly in time with $S^+$ being stronger only in the untapped direction. For a solution controlled in the $S^+$ norm, the difference between\;$U(t)$ and its limit dynamics given by~\eqref{RSS} decays in $S$ norm. This was also the strategy in~\cite{HPTV}; however we point out some differences:
\begin{itemize}
\item[$\bullet$] The Schr\"odinger evolution group associated to the harmonic oscillator enjoys a full range of Strichartz estimates (see~\eqref{Stri0}), which we freely exploit. These are enough to prove Lemma~\ref{Lem32} compared to the  corresponding result~\cite[Lemma 7.1]{HPTV}, where bilinear  Strichartz estimates were needed.
\item[$\bullet$] In many places in~\cite{HPTV} it was convenient to bound the eigenfunctions $(e^{in\cdot x})_{n\in \Z^{d}}$ of $\Delta_{\T^{d}}$ by 1, and to use that $\int_{\T^{d}} e^{in_{1}\cdot x}\ov{e^{in_{2}\cdot x}}e^{in_{3}\cdot x}\ov{e^{in_{4}\cdot x}}dx=\delta_{n_{1}-n_{2}+n_{3}-n_{4}}$, which in turn induces a convolution structure in many estimates (see~\cite[estimate (2.14)]{HPTV}). Such things do not hold true for the harmonic oscillator and the Hermite functions. To tackle this issue, we often resort to the Lens transform (see for example~\cite{Tao}) and use linear and bilinear Strichartz estimates (see {\it e.g.} the proof of Lemma~\ref{BilEf} and Lemma~\ref{ddBE}).
\end{itemize}

\subsection{Plan of the paper} The rest of the paper is organized as follows. In Section~\ref{sect2} we introduce the main notations of the paper and state some preliminary estimates. Section~\ref{sect3} is devoted to the study of the resonant and limit systems. In Section~\ref{sect4} we prove the key result on the structure of the nonlinearity, which will be at the heart of the proof of the main theorems in Section~\ref{sect5}.  Finally, in Section~\ref{sect6} we state and prove a transfer principle which allows us to simplify the proof of many trilinear estimates throughout the paper. 

\subsection{Acknowledgements:} We wish to thank Panayotis Kevrekidis for several illuminating discussions pointing out the relation of this work to the superfluidity and Bose-Einstein condensates in ``cigar-shaped" traps (namely reference~\cite{Frantz}). This work was initiated during the visit of the second author to the Courant Institute of Mathematical Sciences, and he thanks the Institute for its hospitality.

\section{Notations and preliminary estimates}\label{sect2}
\subsection{Standard notations} In this paper, $\N$ denotes the set of all non-negative integers, and~$\Z$ is the set of all integers.
We will often consider functions $f:\mathbb{R}\to\mathbb{C}$ and functions ${F:\mathbb{R}\times \mathbb{R}^d\to\mathbb{C}}$, which we distinguish by using the convention that lower case letters denote functions defined on $\mathbb{R}$, capitalized letters denote functions defined on $\mathbb{R}\times\mathbb{R}^d$.
\medskip

We denote by   $\HH_{d}$ the harmonic oscillator  in dimension $d$. The operator  $\HH_{d}$ admits a Hilbertian  basis of eigenvectors for $L^2(\R^{d})$. We will denote the $n-$th eigenspace by $E_n$ with $n \in \N$ and the corresponding eigenvalue by $\lambda_n=2n+d$. Denote by $K_{n}$ the dimension of $E_{n}$. Then  $K_{n}=\# \{(k_{1},\dots, k_{d}) \in \N^{d}\;:\; k_{1}+\cdots+ k_{d}=n\}$ and we can   check that   $K_{n}\sim c_{d}n^{d-1}$. Each eigenspace $E_n$ is spanned by the Hermite functions $(\psi_{n,j})_{1\leq j\leq K_{n}}$. For more details,   we refer to Helffer~\cite{Helffer} or to the course of Ramond~\cite{Ramond}.
\medskip

The scale  of harmonic Sobolev spaces is  defined as follows:
$s\geq 0$, $p\geq 1$. 
 \begin{equation*} 
         \WW^{s, p}= \WW^{s, p}(\R^d) = \big\{ u\in L^p(\R^d),\; \HH_{d}^{s/2}u\in L^p(\R^d)\big\},     
       \end{equation*}
       \begin{equation*}
           \HH^{s}=   {\mathcal H}^{s}(\R^d) = \WW^{s, 2}.
       \end{equation*}
             The natural norms are denoted by $\Vert u\Vert_{\WW^{s,p}}$ and up to equivalence of norms we have (see~\cite[Lemma~2.4]{YajimaZhang2}) 
             for $1<p<+\infty$
             \begin{equation}\label{eq}
                 \Vert u\Vert_{\WW_{y}^{s,p}} = \Vert  \HH_{d}^{s/2}u\Vert_{L_{y}^{p}} \equiv \Vert (-\Delta)^{s/2} u\Vert_{L_{y}^{p}} + 
       \Vert\<y\>^{s}u\Vert_{L_{y}^{p}}.
             \end{equation}

In all the paper, we use the notation 
\begin{equation*}
\DD=-\partial^2_x+\mathcal{H}_{d}.
\end{equation*}

Recall that we have defined the Fourier transform on $\mathbb{R}$ by
\begin{equation*}
\mathcal{F}_{x}g(\xi)=\widehat{g}(\xi):=\frac{1}{2\pi}\int_{\mathbb{R}}e^{-ix\xi}g(x)dx.
\end{equation*}
Similarly, if $F(x,y)$ depends on $(x,y)\in\mathbb{R}\times\mathbb{R}^{d}$, $\widehat{F}(\xi,y)$ denotes the partial Fourier transform in~$x$.  Denote by  $\Pi_n$ the projection onto the eigenspace $E_n$ and $h_{n,j}:=\langle h, \psi_{n,j}\rangle$. We then  consider the Hermite expansion of functions $h:\mathbb{R}^{d}\to\mathbb{C}$,
\begin{equation*}
h(y)=\sum_{n\geq 0} h_n(y) \qquad h_n:=\Pi_n(h)=\sum_{j=1}^{K_{n}} h_{n,j}\psi_{n,j}=\sum_{j} h_{n,j}\psi_{n,j},
\end{equation*}
with the convention that $\psi_{n,j}=0$ if $j\geq K_n$. For a bounded function $\varphi$ we define
\begin{equation*}
\varphi(H)h=\sum_{n\geq 0}\varphi(\lambda_{n})h_{n}.
\end{equation*}

The full frequency expansion of $F$ reads
\begin{equation*}
F(x,y)=  \sum_{n \in \N} \int_{\R} \widehat F_n(\xi, y) e^{i\xi x}d\xi= \sum_{n \in \N}\sum_{j=1}^{K_{n}} \int_{\R} \widehat F_{n,j}(\xi) e^{i\xi x} \psi_{n,j}(y)d\xi\, .
\end{equation*}

We will often use Littlewood-Paley projections. For the full frequency space, these are defined as follows:
\begin{equation*}
\begin{split}
\left(\mathcal{F}_{x}P_{\leq N}F\right)(\xi,y)=\varphi(\frac{\xi}{N})\varphi(\frac{\mathcal H}{N^2})\widehat F_p(\xi,y),
\end{split}
\end{equation*}
where $\varphi\in \mathcal{C}^\infty_c(\mathbb{R})$, $\varphi(x)=1$ when $\vert x\vert\leq 1$ and $\varphi(x)=0$ when $\vert x\vert\geq 2$.

In addition, we also define
\begin{equation*} 
\phi(x)=\varphi(x)-\varphi(2x)
\end{equation*}
and
\begin{equation*}
P_N=P_{\leq N}-P_{\leq N/2},\quad P_{\geq N}=1-P_{\leq N/2}.
\end{equation*}

With this notation, the full Sobolev $H^s$   norm on function on $\R\times \R^d$ takes the form
\begin{equation*} 
\|F\|_{H^s(\R\times \R^d)}=\Big(\sum_{N \in 2^{\mathbb N_0}} N^{2s}\|P_N F\|_{L_{x,y}^2}^2\Big)^{1/2}.
\end{equation*}

Many times we concentrate on the frequency in $x$ only, and we therefore define
\begin{equation*}
\begin{split}
\left(\mathcal{F}_{x}Q_{\leq N}F\right)(\xi,y)=\varphi(\frac{\xi}{N})\left(\mathcal{F}_{x}F\right)(\xi,y),
\end{split}
\end{equation*}
and define $Q_N$ similarly. By a slight abuse of notation, we will consider $Q_N$ indifferently as an operator on functions defined on $\mathbb{R}\times\R^{d}$ and on $\mathbb{R}$.
With the Parseval formula, it is easy to check that  
\begin{equation}\label{comm}
\big\|[Q_N,x]\big\|_{L^2_x\rightarrow L^2_x}\lesssim N^{-1}\,.
\end{equation}

Denote by $O_{N}$ the frequency localization in $y$ only. It is then easy to check that the full Sobolev norm $H^{s}(\R\times \R^{d})$ also reads
\begin{equation}\label{Full Sobolev2}
\|F\|_{H^s(\R\times \R^d)}=\Big(\sum_{M,L \in 2^{\mathbb N_0}} (M^{2s}+N^{2s})\|Q_{M}O_{L} F\|_{L_{x,y}^2}^2\Big)^{1/2}.
\end{equation}

\medskip

For $\omega\in 2\Z$, we will use the following sets corresponding to momentum and resonance level sets:
\begin{equation*} 
\begin{split}
\Gamma_\omega&:=\{(p,q,r,s)\in\N:\,\,\lambda_p-\lambda_q+\lambda_r-\lambda_s=2(p-q+r-s)=\omega\}.\\
\end{split}
\end{equation*}
\subsection{The nonlinearity}\label{sect22}
We will prove all our statements for $t\geq 0$. By time-reversal symmetry, one obtains the analogous claims for $t\leq 0$.
We shall consider only the case $\kappa_0=1$ in~\eqref{CNLS}. Since we only deal with small data the case $\kappa_0=-1$ can be treated similarly. \ligne

Let us define the trilinear form $\mathcal{N}^{t}$  by 
\begin{equation}\label{defN}
\mathcal{N}^{t}[F,G,H]:=e^{-it\DD}\Big( e^{it\DD}F\cdot e^{-it\DD}\overline{G}\cdot e^{it\DD}H\Big).
\end{equation}
Let 
\begin{equation*}
U(t,x,y)=e^{it\DD} F(t),
\end{equation*}
then we see that $U$ solves~\eqref{CNLS} if and only if $F$ solves
\begin{equation*} 
i\partial_t F(t) = \mathcal{N}^{t}[F(t),F(t),F(t)].
\end{equation*}
The following formulation holds true
\begin{equation*}
\mathcal{F}_x\mathcal{N}^{t}[F,G,H](\xi,\cdot)
=e^{-i\tau\HH_{d}}\int_{\mathbb{R}^2}e^{2it\eta\kappa} \widehat{e^{it\HH_{d}} F}(\xi-\eta)\overline{\widehat{e^{it\HH_{d}} G}}(\xi-\eta-\kappa)\widehat{e^{it\HH_{d}} H}(\xi-\kappa)d\eta d\kappa .
\end{equation*}
Actually, set 
\begin{equation} \label{defI}
\mathcal{I}^{t}[f,g,h]:=\UU(-t)\Big(\UU(t)f\, \overline{\UU(t)g}\,\UU(t)h \Big),\quad \UU(t)=\exp(-it\partial_x^2),
\end{equation}
then $\dis\mathcal{N}^{t}[F,G,H]=e^{-it\HH_{d}} \mathcal{I}^{t}[e^{it\HH_{d}}F,e^{i\tau\HH_{d}}G,e^{it\HH_{d}}H]$ with 
\begin{equation*}
\mathcal{F}_x\mathcal{I}^{t}[f,g,h](\xi)=\int_{\mathbb{R}^2}e^{2it\eta\kappa} \widehat{f}(\xi-\eta)\overline{\widehat{g}}(\xi-\eta-\kappa)\widehat{h}(\xi-\kappa)d\eta d\kappa.
\end{equation*}

By expanding $F=\sum_{q\geq 0}F_q$ according to the eigenspaces of the harmonic oscillator, and for $G,H$ as well, one may also write
\begin{equation*}
\mathcal{F}_x\mathcal{N}^{t}[F,G,H](\xi,y)
=\sum_{\omega\in 2\Z}e^{-it\omega}\sum_{(p,q,r,s)\in \Gamma_{\omega}} \Pi_{p}\int_{\mathbb{R}^2}e^{2it\eta\kappa} \widehat{  F}_{q}(\xi-\eta,y)\overline{\widehat{ G}_{r}}(\xi-\eta-\kappa,y)\widehat{  H}_{s}(\xi-\kappa,y)d\eta d\kappa .
\end{equation*}
According to our previous discussion, we now define  
the resonant part of the nonlinearity 
as
\begin{equation}\label{DefOfR}
\mathcal{F}_x\mathcal{R}[F,G,H](\xi,y):=\sum_{(p,q,r,s)\in\Gamma_0}\Pi_p\left(\widehat{F}_q(\xi,y)\overline{\widehat{G}_r}(\xi,y)\widehat{H}_s(\xi,y)\right).
\end{equation}
  
\subsection{The norms}
It will be convenient to use the following norms when dealing with functions on\;$\R^{d}$
\begin{equation*}
\Vert f\Vert_{\mathcal H_{y}^s}:=\|\mathcal H_{d}^{s/2} f\|_{L^2(\R^d)}\approx \Big( \sum_{p\in \N}(1+p)^{s} \sum_{j} \vert f_{p, j}\vert^2\Big)^{1/2}.
\end{equation*}
For functions, we will often omit the domain of integration from the description of the norms. However, we will indicate it by a subscript $x$ (for $\mathbb{R}$), $x,y$ (for $\mathbb{R}\times\mathbb{R}^d$). We will use mainly three different norms: a weak norm
\begin{equation*}
\Vert F\Vert_{Z}^2:=\sup_{\xi\in\mathbb{R}}\left[1+\vert \xi\vert^2\right]^2\sum_{p}(1+p) \Vert\widehat{F}_{p}(\xi)\Vert_{L^2_y}^2=\sup_{\xi\in\mathbb{R}}\left[1+\vert \xi\vert^2\right]^2\Vert \widehat{F}(\xi)\Vert_{\mathcal H_y^1}^2
\end{equation*}
and two strong norms
\begin{equation}\label{DefSNorm}
\begin{split}
\Vert F\Vert_{S}:=&\Vert F\Vert_{H^N_{x,y}}+\Vert xF\Vert_{L^2_{x}L^2_{y}},\quad
\Vert F\Vert_{S^+}:=\Vert F\Vert_{S}+\Vert (1-\partial_{xx})^4F\Vert_{S}+\Vert xF\Vert_{S}.
\end{split}
\end{equation}

Note that the operators $Q_{\leq N}$, $P_{\leq N}$ and the multiplication by $\varphi(\cdot/N)$ are bounded in $Z$,\;$S$ and\;$S^+$, uniformly in $N$. Elementary considerations (see Lemma~\ref{ZweakerS} below) show that the following holds
\begin{equation*} 
\Vert F\Vert_{H^1_{x,y}}\lesssim \Vert F\Vert_{Z}\lesssim\Vert F\Vert_{S}\lesssim \Vert F\Vert_{S^+}.
\end{equation*}

Finally, we are ready to fix the space-time norms we will use:   Fix once and for all $\delta<10^{-3}$ and define
\begin{equation*} 
\begin{split}
\Vert F\Vert_{X_T}:=&\sup_{0\leq t\leq T}\big\{\Vert F(t)\Vert_{Z}+(1+\vert t\vert)^{-\delta}\Vert F(t)\Vert_{S} +(1+\vert t\vert)^{1-3\delta}\Vert\partial_t F(t)\Vert_{S}\big\},\\
\Vert F\Vert_{X_T^{+}}:=&\Vert F\Vert_{X_T}+\sup_{0\leq t\leq T}\big\{(1+\vert t\vert)^{-5\delta}\Vert F(t)\Vert_{S^+}+(1+\vert t\vert)^{1-7\delta}\Vert\partial_t F(t)\Vert_{S^+}\big\}.
\end{split}
\end{equation*}
 
\subsection{Preliminary Estimates}
We gather here some useful preliminary estimates that will be useful later in our work.
\begin{lemma}\label{ZweakerS}
\begin{enumerate}
\item  Let $1/2<\a\leq 1$ and $f\in L^{1}(\R)$. Then 
\begin{equation}\label{L11}
\|f\|_{L^1(\mathbb{R})}\lesssim \|f\|_{L^2(\R)}^{1-\frac{1}{2\a}}\| |x|^{\a}  f\|_{L^2(\R)}^{\frac{1}{2\a}}.
\end{equation}

\item Assume that $N\geq 8$ and $F: \R\times \R^d \to \C$. Then the following estimate holds true
\begin{equation}\label{ZSNorm}
 \Vert F\Vert_{Z}\lesssim \Vert F\Vert_{L^2_{x,y}}^\frac{1}{8}\Vert F\Vert_{S}^\frac{7}{8}.
\end{equation}
\end{enumerate}
\end{lemma}

\begin{proof}
(1) Let $R>0$ and $\a>1/2$. By Cauchy-Schwarz we have
\begin{eqnarray*}
\|f\|_{L^1(\mathbb{R})}&=& \int_{|x|<R}|f|+ \int_{|x|>R}|x|^{-\a}|x|^{\a} |f|\\
&\leq & C R^{1/2}\big\|f\big\|_{L^{2}(\R)}+C R^{1/2-\a}\big\||x|^{\a}f\big\|_{L^{2}(\R)}.
\end{eqnarray*}
We then get the result with the choice $R^{\a}=\big\| |x|^{\a}  f\big\|_{L^2(\R)}/\big\|   f\big\|_{L^2(\R)}$.

(2) By part (1) with $\a=1/2$ and~\eqref{comm} we have
\begin{equation*}
\begin{split}
\left[1+\vert \xi\vert^2\right]\vert \widehat{F}_{p}(\xi,\cdot)\vert&
\lesssim\sum_N N^2\vert \widehat{Q_NF_{p}}(\xi)\vert  \lesssim \sum_NN^2\Vert Q_NF_{p}\Vert_{L_x^1}\\ 
&\lesssim \sum_NN^2\Vert Q_NF_{p}\Vert_{L_x^2}^\frac{1}{2}\Vert  x  Q_NF_{p}\Vert_{L_x^2}^\frac{1}{2} \\
&\lesssim \sum_NN^{-\frac{1}{2}}\Vert (1-\partial_{xx})^\frac{5}{2}F_{p}\Vert_{L_x^2}^\frac{1}{2}\Vert x F_{p}\Vert_{L_x^2}^\frac{1}{2}\lesssim \Vert F_{p}\Vert_{H_x^5}^\frac{1}{2}\Vert   x  F_{p}\Vert_{L_x^2}^\frac{1}{2}.
\end{split}
\end{equation*}
Squaring the previous estimate and multiplying  by $\langle p\rangle$, one obtains
\begin{equation*}
\|F\|_{Z}\lesssim \Vert xF\Vert^{1/2}_{L^2_{x,y}}\Vert F\Vert^{1/2}_{H^6_{x,y}}\lesssim  \Vert F\Vert_{L^2_{x,y}}^\frac{1}{8}\Vert F\Vert_{S}^\frac{7}{8},
\end{equation*}
where we concluded by interpolation, using that   $N\geq 8$.
\end{proof}

We now give some crude estimates on $\NN^{t}$ in the $S$ and $S^{+}$ norms.
\begin{lemma}\label{wup}
Let $N\geq 4$, then the following estimates hold true:
\begin{equation*} 
\begin{split}
\Vert \mathcal{N}^{t}[F,G,H]\Vert_{S}&\lesssim (1+\vert t\vert)^{-1}\Vert F\Vert_{S}\Vert G\Vert_{S}\Vert H\Vert_{S},\\[3pt]
\Vert \mathcal{N}^{t}[F^{a},F^{b},F^{c}]\Vert_{S^+}&\lesssim (1+\vert t\vert)^{-1}\max_{\sigma\in\mathfrak{S}_3}\Vert F^{\sigma(a)}\Vert_{S^+}\Vert F^{\sigma(b)}\Vert_{S}\Vert F^{\sigma(c)}\Vert_{S}.
\end{split}
\end{equation*}
\end{lemma}

\begin{proof}
By  Lemma~\ref{ControlSS+}, it is sufficient to prove that 
\begin{equation}\label{SuffLem}
\Vert \mathcal{N}^{t}[F^{a},F^{b},F^{c}]\Vert_{L^2_{x,y}}\lesssim (1+\vert t\vert)^{-1}\min_{\sigma\in\mathfrak{S}_3}\Vert F^{\sigma(a)}\Vert_{L^2_{x,y}}\Vert F^{\sigma(b)}\Vert_{S}\Vert F^{\sigma(c)}\Vert_{S},
\end{equation}
and by symmetry, we only consider the case $\s=id$. Let $G\in L^2_{x,y}$, then by~\eqref{defN}
\begin{eqnarray*}
\<    \mathcal{N}^{t}[F^{a},F^{b},F^{c}],G     \>_{L^2_{x,y}}&=&\int_{x,y}   \big(e^{it\DD}F^{a}\big) \big( e^{-it\DD}\overline{F^{b}}\big) \big(e^{it\DD}F^{c}\big) \big(e^{-it\DD}\ov{G} \big)dxdy\\
&\leq &   \big\Vert  F^{a}\big\Vert_{L^2_{x,y}} \big\Vert e^{it\DD}F^{b}\big\Vert_{L^{\infty}_{x,y}}\big\Vert e^{it\DD}F^{c}\big\Vert_{L^{\infty}_{x,y}}\big\Vert  G\big\Vert_{L^2_{x,y}}.
\end{eqnarray*}
$\bullet$ Assume $\vert t\vert\geq 1$. For $F=F^{a},F^{b}$  we show that
\begin{equation}\label{estdisp}
 \big\Vert e^{it\DD}F\big\Vert_{L^{\infty}_{x,y}}  \lesssim\vert t\vert^{-\frac{1}{2}}  \big\Vert F\big\Vert_{S}. 
\end{equation}
We write 
\begin{equation*}
F(x,y)=\sum_{q,j} F_{q,j}(x)\psi_{q,j}(y),
\end{equation*}
thus 
\begin{equation*}\label{series}
e^{it\DD}F(x,y)=\sum_{q,j}e^{it\lambda_q}e^{-it\partial^2_x} F_{q,j}(x)\psi_{q,j}(y).
\end{equation*}
Then, we use the basic dispersive bound for the $1d$ Schr\"odinger equation and~\eqref{L11} to get 
 \begin{equation*} 
 \big\Vert e^{-it\partial^2_x} F_{q,j}\big\Vert_{L^{\infty}_{x}} 
\lesssim\vert t\vert^{-\frac{1}{2}} \big\Vert  F_{q,j}\big\Vert_{L^{1}_{x}}
\lesssim\vert t\vert^{-\frac{1}{2}}
 \big\Vert  F_{q,j}\big\Vert^{1/2}_{L^{2}_{x}} \big\Vert x F_{q,j}\big\Vert^{1/2}_{L^{2}_{x}} \,.
\end{equation*}
By  Thangavelu/Karadzhov (see~\cite[Lemma 3.5]{PRT1}), for all $d\leq 4$ and $q\geq 1$ 
\begin{equation*} 
  \sup_{y\in \R} \sum_{j} | \psi_{q,j}\big(y\big)\big|^{2}\leq C q.
\end{equation*}
 By~\eqref{series}, the previous lines and Cauchy-Schwarz, we get
 \begin{eqnarray*}
  \big\Vert e^{it\DD}F(.,y)\big\Vert_{L^{\infty}_{x}}&  \lesssim&\sum_{q,j}  \big\Vert e^{-it\partial^2_x} F_{q,j}\big\Vert_{L^{\infty}_{x}} \vert \psi_{q,j}(y)\vert \\
  &  \lesssim&\sum_{q} \big(\sum_{j}  \big\Vert e^{-it\partial^2_x} F_{q,j}\big\Vert^2_{L^{\infty}_{x}}\big)^{1/2}   \big(\sum_{j}  \vert \psi_{q,j}(y)\vert^2 \big)^{1/2}\\
    &  \lesssim& \vert t\vert^{-\frac{1}{2}} \sum_{q} q^{1/2}\big(\sum_{j}   \big\Vert  F_{q,j}\big\Vert_{L^{2}_{x}} \big\Vert x F_{q,j}\big\Vert_{L^{2}_{x}}\big)^{1/2}   \\
        &  \lesssim& \vert t\vert^{-\frac{1}{2}} \sum_{q} q^{1/2}\big(\sum_{j}   \big\Vert  F_{q,j}\big\Vert^2_{L^{2}_{x}}\big)^{1/4}  \big(\sum_j \big\Vert x F_{q,j}\big\Vert^2_{L^{2}_{x}}\big)^{1/4}.  
 \end{eqnarray*}
  Now, for $\beta>1$, and Cauchy-Schwarz again 
 \begin{eqnarray}\label{estFt}
  \big\Vert e^{it\DD}F\big\Vert_{L^{\infty}_{x,y}}&  \lesssim&  \vert t\vert^{-\frac{1}{2}} \sum_{q} q^{-\beta/2}\big(\sum_{j}   q^{2(\beta+1)}\big\Vert  F_{q,j}\big\Vert^2_{L^{2}_{x}}\big)^{1/4}  \big(\sum_j \big\Vert x F_{q,j}\big\Vert^2_{L^{2}_{x}}\big)^{1/4}  \nonumber\\
  &  \lesssim& \vert t\vert^{-\frac{1}{2}}    \big\Vert F\big\Vert^{1/2}_{L^{2}_{x}\HH^{N}}    \big\Vert xF\big\Vert^{1/2}_{L^{2}_{x,y}} \lesssim \vert t\vert^{-\frac{1}{2}} \big\Vert F\big\Vert_{S},
 \end{eqnarray}
 for $N\geq 4$, which was the claim.  
 
$\bullet$ Assume $\vert t\vert\leq 1$, then by Sobolev, we clearly have $ \big\Vert e^{it\DD}F^{b}\big\Vert_{L^{\infty}_{x,y}}  \lesssim  \big\Vert F^{b}\big\Vert_{S}. $
\end{proof}
 
 Finally, let us recall the following result from~\cite[Lemma 7.5]{HPTV}
\begin{lemma}\label{CM}
Let $ \dis \frac{1}{p}=\frac{1}{q}+\frac{1}{r}+\frac{1}{s}$ with $ 1\leq p,q,r,s\le\infty$,
then
\begin{equation*}
\begin{split}
\big\Vert \int_{\mathbb{R}^3}e^{ix\xi}m(\xi,\eta,\kappa)\widehat{f}(\xi-\eta)\overline{\widehat{g}}(\xi-\eta-\kappa)\widehat{h}(\xi-\kappa)d\eta d\kappa d\xi\big\Vert_{L_{x}^p}\lesssim\Vert\mathcal{F}^{-1}m\Vert_{L^1(\mathbb{R}^3)}\Vert f\Vert_{L^q}\Vert g\Vert_{L^r}\Vert h\Vert_{L^s}.
\end{split}
\end{equation*}
\end{lemma}

\begin{proof}
Denote by $(t,u,v)$ the dual variables of $(\xi,\eta,\kappa)$ and for define $t\mapsto f_{u}(t)=f(t-u)$. Then by Parseval for all $x\in \R$
\begin{eqnarray*}
 I(x)&:=&\int_{\mathbb{R}^3}m(\xi,\eta,\kappa)e^{ix\xi}\widehat{f}(\xi-\eta)\overline{\widehat{g}}(\xi-\eta-\kappa)\widehat{h}(\xi-\kappa)d\eta d\kappa d\xi \\
 &=&\int_{\mathbb{R}^3} (\ov{\mathcal{F}^{-1}\ov{m})}(t,u,v) \Big[ \int_{\R}  e^{i\xi(x+t+u+v)}\Big( \int_{\R^2}\widehat{f_{u}}(\xi-\eta)\overline{\widehat{g}}(\xi-\eta-\kappa)\widehat{h_{v}}(\xi-\kappa) d\eta d\kappa\Big) d\xi\Big]dtdudv\\
 &=&\int_{\mathbb{R}^3} (\ov{\mathcal{F}^{-1}\ov{m})}(t,u,v) \big( ( f_{u}\,\ov{g}\,h_{v})(x+t+u+v)\big)dtdudv,
\end{eqnarray*}
where $\mathcal{F}^{-1}$ stands here for the Fourier transform all the variables. Thus by the H\"older inequality
\begin{eqnarray*}
 \|I\|_{L^{p}_{x}}&\leq &\int_{\mathbb{R}^3} |(\mathcal{F}^{-1}m)(t,u,v)| \star\| fgh\|_{L^{p}_{x}}dtdudv\\
 &=&      \Vert\mathcal{F}^{-1}m\Vert_{L^1(\mathbb{R}^3)}\Vert fgh\Vert_{L^p}\\
 &\leq &  \Vert\mathcal{F}^{-1}m\Vert_{L^1(\mathbb{R}^3)}\Vert f\Vert_{L^q}\Vert g\Vert_{L^r}\Vert h\Vert_{L^s},
\end{eqnarray*}
which was to be proved.
\end{proof}
 

\section{The resonant system}\label{sect3}
In this section, we study the asymptotic system~\eqref{RSS} in terms of its well-posedness and stability properties. All these properties are directly inherited from those of the resonant system~\eqref{RS}. We also obtain some bounds on the growth of solutions in the $S$ and $S^+$ norms. In the case $d=2$, we point out the relation to the (CR) equation derived in~\cite{FGH} and recall some of its stationary solutions~\cite{FGH, GHT1}. These solutions will give us interesting dynamics for the limit system~\eqref{RSS} and hence for~\eqref{CNLS}. 

\subsection{Conserved quantities and well-posedness} The resonant system~\eqref{RS} is Hamiltonian with an energy functional given by
\begin{equation}\label{RSHam}
Q[g]=\sum_{\substack{n_{1},n_{2},n_3,n_4 \in \N\\n_{1}-n_{2}+ n_3=n_4}} \int_{\R^d} g_{n_{1}}\ov{g_{n_{2}}}g_{n_3}\overline{g_{n_4}}\,dy
\end{equation}
under the standard symplectic structure on $L^2(\R^d)$ given by $ \Im \int f \bar g\, dy$. By noticing that the integral above vanishes unless $n_1-n_2+n_3-n_4$ is even (since $g_k$ has the same parity as $k$), we can rewrite $Q$ as:
\begin{eqnarray*}
Q[g]&=&\frac{2}{\pi}\sum_{\substack{n_{1},n_{2},n_3,n_4 \in \N}}\int_{-\pi/4}^{\pi/4} e^{2i\lambda (n_{1}-n_{2}+ n_3-n_4)}\int_{\R^d} g_{n_{1}}\ov{g_{n_{2}}}g_{n_3}\overline{g_{n_4}}\,dy \;d\lambda\\
&=&\frac{2}{\pi}\int_{-\pi/4}^{\pi/4}\int_{\R^d} |e^{i\lambda \mathcal H} g|^4 \,d\lambda\, dy\nonumber
=\frac{2}{\pi}\left\| e^{i\lambda \mathcal H} g\right\|^4_{L^4_{\lambda,y}([-\frac{\pi}{4}, \frac{\pi}{4}]\times \R^d)}.  
\end{eqnarray*}
Thus, the Hamiltonian of the resonant system is none other than the forth power of the $L^4$ Strichartz nom associated with the trapped Schr\"odinger equation $(i\partial_t +\mathcal H)u=0$. 
\medskip

Next, we notice that the resonant system~\eqref{RS} conserves mass and kinetic energy given respectively by:
\begin{equation*} 
M[g]:=\int_{\R^d} |g(y)|^2\, dy;\qquad K.E.[g]=\int_{\R^d} \HH_{d} g\, \bar g \, dx=\sum_{p}\lambda_p \|g_p\|_{L^2(\R^d)}^2=\|g\|_{\HH_y^1}^2.
\end{equation*}
This follows from noticing that 
\begin{align*}
\partial_t \Big(\sum_{p} a(p) \|g_p\|_{L^2}^2\Big)=&\Im \sum_{q-r+s=p}a(p) \int_{\R^d} g_{q} \overline{g_r}g_s \overline{g_p} \,dy\\
=&\frac 14 \Im  \sum_{q-r+s=p}\big(a(q)-a(r)+a(s)-a(p)\big)  \int_{\R^d} g_{q} \overline{g_r}g_s \overline{g_p} \,dy
\end{align*}
which vanishes if $a(p)$ is an affine function.
 This conservation of $\HH^1_y$ norm allows for the following global well-posedness result

\begin{lemma} 
Let $1\leq d\leq 4$. For any $ G(0)\in \HH^1_y(\R^{d})$, there exists a unique global solution $u\in \mathcal{C}^1\big(\mathbb{R}:\HH^1_y(\R^{d})\big)$ of~\eqref{RS}.
In addition, higher regularity is preserved in the sense that if $G(0)\in \HH^s_y(\R^{d})$, then the solution belongs to $\mathcal{C}^1(\mathbb{R}:\HH^s_y(\R^{d}))$.
\end{lemma}

This result follows from the standard arguments combining the conservation laws with estimate~\eqref{StriEst} below, which gives local well-posedness of~\eqref{RSS} in all Sobolev spaces $H^\sigma(\R^d)$ for all $\sigma\geq 1$ and $1\leq d\leq 4$.

\begin{lemma}\label{Lem32}
Let $\TT$ be defined by~\eqref{RS}. For every function $G_{1}, G_{2}$ and $G_{3}$ we have
\begin{equation}\label{StriEst2}
\Vert \TT[G_{1},G_{2},G_3]\Vert_{L^2_y}\lesssim 
\min_{\tau\in\mathfrak{S}_3}\Vert G_{\tau(1)}\Vert_{L^2_y}\Vert G_{\tau(2)}\Vert_{\HH^{1}_{y}}\Vert G_{\tau(3)}\Vert_{\HH^1_y},
\end{equation}
and consequently, for any $\sigma\geq 0$,
\begin{equation}\label{StriEst}
\Vert \TT[G_{1},G_{2},G_3]\Vert_{\HH^\sigma_y}\lesssim 
\sum_{\tau\in\mathfrak{S}_3}\Vert G_{\tau(1)}\Vert_{\HH^\sigma_y}\Vert G_{\tau(2)}\Vert_{\HH^1_y}\Vert G_{\tau(3)}\Vert_{\HH^1_y}.
\end{equation}
\end{lemma}

\begin{proof}[Proof of Lemma~\ref{Lem32}]

One can deduce~\eqref{StriEst} from~\eqref{StriEst2} using Lemma~\ref{ControlSS+}.

We prove~\eqref{StriEst2}. We start by recalling the Strichartz  estimates for the harmonic oscillator (see~\cite{Poiret} for a proof).  A couple $(q,r)\in [2,+\infty]^2$ is called admissible if 
\begin{equation*}
\frac2q+\frac{d}{r}=\frac{d}2\quad \text{and}\quad (d,q,r)\neq (2,2,+\infty),
\end{equation*}
and if one defines 
\begin{equation*}
 {X}^s:= \bigcap_{(q,r) \ admissible} L ^q\big( [0,T] \,; \WW^{s,r}( \R^d )\big),
\end{equation*}
then for all $T>0$ there exists $C>0$ so that for all $G\in \HH^{s}(\R^{d})$ we have 
\begin{equation}\label{Stri0}
\|e^{-itH}G\|_{X^{s}}\leq C\|G\|_{\HH^{s}(\R^{d}).}
\end{equation}

Now, let  $G_{0}\in L^{2}(\R^{d})$, and define $u_{j}(t,y)=e^{it\HH_{d}}G_{j}(y)$, then we have
\begin{equation}\label{95}
  \<    \TT[G_{1},G_{2},G_3], G_{0}  \>=\frac1{\pi}\int_{0}^{\pi}\int_{\R^{d}}u_{1}(t,y)\overline{u_{2}(t,y)}u_3(t,y)\overline{u_{0}(t,y)}\,dydt\,.
\end{equation}
By~\eqref{Stri0}, we get that for $d\leq 4$, $\|e^{it\HH_{d}} G\|_{L^{2}_{[0,\pi]}L_{y}^{4}}\leq C \|G\|_{L_{y}^{2}}$.  Next, for $d\leq 4$ we have the Sobolev embedding $\mathcal{W}^{1,8/3}(\R^{d})\subset L^{8}(\R^{d})$, and by~\eqref{Stri0} we have 
\begin{equation*}
 \|e^{it\HH_{d}} G\|_{L^{4}_{[0,\pi]}L_{y}^{8}}\leq C \|e^{it\HH_{d}} G\|_{L^{4}_{[0,\pi]}\WW_{y}^{1,8/3}}\leq C\|  G\|_{\HH_{y}^{1}}.
\end{equation*}
Now from~\eqref{95},  H\"older and the previous estimates we get 
\begin{eqnarray*}
\big|  \<    \TT[G_{1},G_{2},G_3], G_{0}  \> \big| &\leq& C \|u_{0}\|_{L^{\infty}_{[0,\pi]}L_{y}^{2}}\|u_{1}\|_{L^{2}_{[0,\pi]}L_{y}^{4}}\|u_{2}\|_{L^{4}_{[0,\pi]}L_{y}^{8}}\|u_{3}\|_{L^{4}_{[0,\pi]}L_{y}^{8}}\\
&\leq &C\Vert G_{0}\Vert_{L^2_y} \Vert G_{1}\Vert_{L^2_y}\Vert G_{2}\Vert_{\HH^{1}_{y}}\Vert G_{3}\Vert_{\HH^1_y},
\end{eqnarray*}
which was the claim.
\end{proof}

As a consequence of the above lemma, the corresponding estimates for the limit system follow directly.

\begin{corollary} 
Let $\RR$ be defined by~\eqref{DefOfR}. Then for every function $G_{1}, G_{2}$ and $G_{3}$  the following estimates hold true
\begin{align}
\Vert \RR[G_{1},G_{2},G_3]\Vert_{L^2_{x,y}}\;\lesssim\;& 
\min_{\tau\in\mathfrak{S}_3}\Vert G_{\tau(1)}\Vert_{L^2_{x,y}}\Vert G_{\tau(2)}\Vert_{Z}\Vert G_{\tau(3)}\Vert_{Z},\label{a0}\\[5pt]
\Vert \RR[G_{1},G_{2},G_3]\Vert_{Z}\;\lesssim\;& 
 \Vert G_{1}\Vert_{Z}\Vert G_{2}\Vert_{Z}\Vert G_{3}\Vert_{Z},\nonumber\\[5pt]
 \Vert \RR[G_{1},G_{2},G_3]\Vert_{S}\;\lesssim\;& 
\max_{\tau\in\mathfrak{S}_3}\Vert G_{\tau(1)}\Vert_{Z}\Vert G_{\tau(2)}\Vert_{Z}\Vert G_{\tau(3)}\Vert_{S}.\label{a2}
\end{align}
\end{corollary}

We are now able to state a result concerning the long time behaviour and the stability of  the asymptotic system~\eqref{RSS}. We refer to~\cite[Lemma 4.3]{HPTV} for the proof.

 \begin{lemma} \label{lem34} The notation $S^{\star}$   stands     either for  $S$ or $S^+$.
\begin{enumerate}[(1)]
\item  Assume that $G_0\in S^{\star}$ and that $G$ evolves according to~\eqref{RSS}. Then, there holds that, for $t\geq 1$,
\begin{equation*} 
\begin{split}
\Vert G(\ln t)\Vert_{Z}&=\Vert G_0\Vert_{Z}\\[3pt]
\Vert G(\ln t)\Vert_{S^{\star}}&\lesssim (1+\vert t\vert)^{\delta^\prime}\Vert G_0\Vert_{S^{\star}}.
\end{split}
\end{equation*}
Besides, we may choose $\delta^\prime\lesssim \Vert G_0\Vert_{Z}^2$.
\item  In addition, we have the following uniform continuity result:
if $A$ and $B$ solve~\eqref{RSS} and satisfy
\begin{equation*}
\sup_{0\leq t\leq T}\left\{\Vert A(t)\Vert_{Z}+\Vert B(t)\Vert_{Z}\right\}\leq \theta
\end{equation*}
and
\begin{equation*}
\Vert A(0)-B(0)\Vert_{S^{\star}}\leq \delta
\end{equation*}
then, there holds that, for $0\leq t\leq T$,
\begin{equation}\label{stab}
\Vert A(t)-B(t)\Vert_{S^{\star}}\leq \delta e^{C\theta^2 t}. 
\end{equation}
\end{enumerate}
\end{lemma}
  
\subsection{Special dynamics of the resonant system}\label{SpeDyn}

In view of Theorem~\ref{ExMWO}, virtually any global dynamic that can isolated for~\eqref{RS} can be transferred to~\eqref{CNLS}. To transfer information from a global solution $f(t)$ of~\eqref{RS} to a solution of~\eqref{RSS}, all one needs to do is take an initial data of the form
\begin{equation*}
G_0(x,y)=\varepsilon \check{\varphi}(x)f(0,y), \qquad \varphi \in \mathcal S(\R).
\end{equation*}
The solution $G(t)$ to~\eqref{RSS} with initial data $G_0$ as above is given in Fourier space by 
$$\widehat G(t, \xi,y)=\varepsilon \varphi(\xi)f(\varepsilon^2\varphi^2(\xi) t, y).$$
In particular, if $\varphi=1$ on an open interval $I$, then $\widehat G(t, \xi,y)=\varepsilon f(\varepsilon^2 t, y)$ for all $t\in \R$ and $\xi \in I$.

\subsection{Relation to the (CR) equation}\label{sect33} When $d=2$, the resonant system~\eqref{RS} is, up to multiplicative constant, the (CR) equation derived in~\cite{FGH} and further studied in~\cite{GHT1, GHT2}. The (CR) equation (standing for {\it continuous resonant}) can be formally derived by starting with the cubic nonlinear Schr\"odinger equation on a periodic box of size $L$ and performing the following two limits\footnote{These limits are motivated in by standard {\it weak turbulence closures}.}: 
\begin{enumerate}
\item {\it Small nonlinearity limit} which amounts to reducing to the resonant dynamics of the cubic NLS equation,
\item {\it Large-box limit} ($L\to \infty$), also known as the infinite volume approximation.
\end{enumerate}
More precisely, if one considers the equation
 \begin{equation} \label{cubicNLS} 
  \left\{
      \begin{aligned}
      &i\partial_t v +\Delta v=\pm  |v|^2v, \quad x\in \T^2_L=[0,L]\times[0, L], \\
     & v(0,x)= \varepsilon v_0(x),\qquad \varepsilon \ll 1,
      \end{aligned}
    \right.
\end{equation}
then in the regime $0<\epsilon < L^{-1-\delta}$ (for any $\delta>0$), the nonlinear dynamics of $\widehat v(t, K)$ can be approximated, over long nonlinear time scales, by $g(t, K)$ where $g(t, \xi): \R_t \times \R_\xi^2 \to \C$ solves the following equation:
\begin{equation}\label{star}
\begin{split}
i\partial_t g(t,\xi)=&\pi^2 \mathcal T(g,g,g)(t,\xi),\qquad \xi \in \R^2,\\
\mathcal T(g,g,g)(t,\xi)=&\frac{1}{\pi^2} \int_{\R}\int_{\R^2} g(t,\xi+\lambda z) \overline{g}(t,\xi+\lambda z+z^\perp)g(t,\xi+z^\perp) \,dz\,d\lambda.
\end{split}
\end{equation}

For the cubic NLS equation on $\T^2$ ($L=2\pi$), equation~\eqref{star} can also be understood as the equation of the high frequency envelopes (profiles) of the NLS solution (cf.~\cite[Theorem~1.2]{FGH}). The rigorous derivation and precise approximation results are contained in~\cite[Theorems 1.1,~1.2]{FGH}. Equation~\eqref{star} has several remarkable properties, of which we mention the following:
\begin{enumerate}
\item {\it Hamiltonian structure:} The equation~\eqref{star} is Hamiltonian and derives from the functional given by the $L^4_{t, x}$ Strichartz norm on $\R^2$ 
$$
E(g)= \iint_{\R\times \R^2} |e^{it\Delta_{\R^2}} g|^4 \, dx\,dt.
$$
\item {\it Invariance of the Fourier transform:} in the sense that if $g(t)$ is a solution of~\eqref{star}, then so is $\widehat g(t)=\mathcal F_{\R^2} g(t)$.
\item {\it Invariance of harmonic oscillator eigenspaces:} The equation~\eqref{star} leaves invariant the eigenspaces of the quantum-harmonic oscillator $-\Delta+|y|^2$.
\item {\it Stationary solutions:} The equation~\eqref{star} admits explicit stationary solutions at each energy level of the harmonic oscillator eigenspaces. In particular,  for all $n\geq 0$, there exists $\lambda_n\in \R$ such that $g(t,y)= e^{i \lambda_n t}  (y_{1}\pm i y_{2})^ne^{-\frac{1}{2}|y|^2}$ solves~\eqref{star} explicitly.
\item \label{pt5}{\it  Quasi-periodic solutions with two frequencies:}  One can describe explicitly the dynamics of the eigenspace $E_{1}=\operatorname{Span}\{  (y_{1}+i y_{2})e^{-\frac{1}{2}|y|^2},  (y_{1}-i y_{2})e^{-\frac{1}{2}|y|^2}\}$. Actually, every solution~$g$ of~\eqref{star} can be written as
\begin{equation*}
g(t,y)=c_+e^{i\lambda_+ t }(y_{1}+i y_{2})e^{-\frac{1}{2}|y|^2}  + c_{-} e^{i\lambda_- t}  (y_{1}-i y_{2})e^{-\frac{1}{2}|y|^2},
\end{equation*}
with $\lambda_+=-\frac{\pi }{4}(\vert c_+\vert^2 +  2\vert c_{-}\vert^2)$ and $\lambda_-=-\frac{\pi }{4}(2\vert c_+\vert^2 +  \vert c_{-}\vert^2)$. Therefore, we can choose the initial conditions such $\lambda_+/\lambda_-=r$ for any $r\in ]0,1[$ (cf.~\cite{GHT1}).
\end{enumerate}
 
\begin{remark}
The fact that the resonant system~\eqref{RS} is the same as~\eqref{star} was observed in~\cite{GHT1}. It is easily seen by noticing that the two systems have equal Hamiltonians (up to constants): recall the definition \eqref{RSHam}, then 
$$
Q[g]=\frac{2}{\pi}\int_{-\pi/4}^{\pi/4}\int_{\R^d} |e^{i\lambda \mathcal H} g|^4 \,d\lambda\, dy=\frac{2}{\pi}\iint_{\R\times \R^2} |e^{it\Delta_{\R^2}} g|^4 \, dx\,dt,
$$
an identity that follows using the Lens Transform~\eqref{LensTransform} and changing variable $x\to \frac{x}{\sqrt{1+4t^2}}$, $ t \to \frac{1}{2}\arctan 2t$.
\end{remark}
\begin{proof}[Proof of Corollary~\ref{coro}] $(1)$: Let $g$ be a solution to~\eqref{RS} such that $\|\partial_{t}g\|_{L^{2}(\R^{2})}\longrightarrow 0$ when $t\longrightarrow +\infty$, then by the conservation of $\|g\|_{L^{2}}$ we get that 
\begin{equation*}
E(g)=\<i \partial_{t}g,g\>_{L^{2}}\longrightarrow 0, \quad \text{when}\quad t\longrightarrow+\infty,
\end{equation*}
then implies that $g\equiv 0$.\\
$(2)$: It is a direct consequence of Theorem~\ref{ExMWO} and point~\eqref{pt5} above.
\end{proof}

\section{Structure of the nonlinearity}\label{sect4}
The purpose of this section is to extract the key effective interactions from the full nonlinearity in~\eqref{CNLS}. Our main result is the following
\begin{proposition}\label{StrucNon}
Decompose the nonlinearity
\begin{equation*} 
\mathcal{N}^{t}[F,G,H]=\frac{\pi}{t}\mathcal{R}[F, G,H]+\mathcal{E}^{t}[F,G,H]\\
\end{equation*}
where $\mathcal{R}$ is given in~\eqref{DefOfR}. 
Assume that for $T^\ast\geq 1$,  $F$, $G$, $H$: $\mathbb{R}\to S$ satisfy
\begin{equation}\label{leq1}
\Vert F\Vert_{X_{T^\ast}}+\Vert G\Vert_{X_{T^\ast}}+\Vert H\Vert_{X_{T^\ast}}\leq 1.
\end{equation}
Then we can write
$$
\mathcal{E}^{t}[F(t),G(t),H(t)]=\mathcal{E}_{1}^{t}[F(t),G(t),H(t)]+\mathcal{E}_{2}^{t}[F(t),G(t),H(t)],
$$
and if for $i=1,2$ we note $\mathcal{E}_i(t):=\mathcal{E}_i^{t}[F(t),G(t),H(t)]$ then the following estimates hold uniformly in $T^\ast\geq 1$,
\begin{equation*} 
\begin{split}
\sup_{1\leq T\leq T^\ast}T^{-\delta}\Vert \int_{T/2}^T\mathcal{E}_i(t)dt\Vert_{S}\lesssim 1,\quad i=1,2,\\
\sup_{1\leq t\leq T^\ast}(1+\vert t\vert)^{1+\delta}\Vert \mathcal{E}_{1}(t)\Vert_{Z}\lesssim 1,\\
\sup_{1\leq t\leq T^\ast}(1+\vert t\vert)^{1/10}\Vert \mathcal{E}_3(t)\Vert_{S}\lesssim 1,
\end{split}
\end{equation*}
where $\mathcal{E}_{2}(t)=\partial_t\mathcal{E}_3(t)$.
Assuming in addition
\begin{equation}\label{BA+}
\Vert F\Vert_{X^+_{T^\ast}}+\Vert G\Vert_{X^+_{T^\ast}}+\Vert H\Vert_{X^+_{T^\ast}}\leq 1,
\end{equation}
we also have that
\begin{equation*} 
\sup_{1\leq T\leq T^\ast}T^{-5\delta}\Vert \int_{T/2}^T\mathcal{E}_i(t)dt\Vert_{S^+}\lesssim 1,\qquad \sup_{1\leq T\leq T^\ast} T^{2\delta}\Vert \int_{T/2}^T\mathcal{E}_i(t)dt\Vert_{S}\lesssim 1,\quad i=1,2.\\
\end{equation*}
\end{proposition}

Actually, in most applications of Proposition~\ref{StrucNon} one can take $T^*=+\infty$. However, there is one place where it is convenient to apply the result with $T^*<+\infty$ : This is in Step 1 of Section~\ref{SMS}, where global existence of $F$ is obtained thanks to a continuity argument (instead of a fixed point in\;$X^+_\infty$). It is there that the argument is non-perturbative and one has to profit from a key cancellation. 
 \subsection{The high frequency estimates}
We start with an estimate bounding high frequencies in $x$. It uses essentially 
the bilinear Strichartz estimates on $\mathbb{R}$ (see Lemma~\ref{1dBE} and~\cite{CKSTTSIMA}).
\begin{lemma}\label{BilEf}
Assume that $ N\geq 8$. The following bounds hold uniformly in $T\geq1$

\begin{equation}\label{est1}
\big\Vert \sum_{\substack{A,B,C\\\max(A,B,C)\geq T^{\frac{1}{6}}}}\mathcal{N}^{t}[Q_{A}F,Q_{B}G,Q_CH]\big\Vert_{Z}\lesssim
 T^{-\frac{201}{200}}\Vert F\Vert_{S}\Vert G\Vert_{S}\Vert H\Vert_{S},\quad\forall t\geq T/4,
 \end{equation}
 \begin{multline}\label{est2}
\Big\Vert \sum_{\substack{A,B,C\\\max(A,B,C)\geq T^{\frac{1}{6}}}}\int_{T/2}^T\mathcal{N}^{t}[Q_{A}F(t),Q_{B}G(t),Q_CH(t)]dt\Big\Vert_{S}\lesssim \\
\lesssim  T^{-\frac{1}{50}}\Vert F\Vert_{X_T}\Vert G\Vert_{X_T}\Vert H\Vert_{X_T},\qquad
 \end{multline}
 \begin{multline}\label{est3}
\Big\Vert \sum_{\substack{A,B,C\\\max(A,B,C)\geq T^{\frac{1}{6}}}}\int_{T/2}^T\mathcal{N}^{t}[Q_{A}F(t),Q_{B}G(t),Q_CH(t)]dt\Big\Vert_{S^+}\lesssim \\
\lesssim T^{-\frac{1}{50}}\Vert F\Vert_{X_T^{+}}\Vert G\Vert_{X_T^{+}}\Vert H\Vert_{X_T^{+}}.\qquad
 \end{multline}
\end{lemma}

\begin{proof}
We prove~\eqref{est1}. Let $t\geq T/4$. By Lemma\;\ref{wup} and estimate~\eqref{ZSNorm} we only have to prove
\begin{equation*}
\Vert \sum_{\substack{A,B,C\\\max(A,B,C)\geq T^{\frac{1}{6}}}}\mathcal{N}^{t}[Q_{A}F,Q_{B}G,Q_CH]\Vert_{L^2_{x,y}}\lesssim T^{-\frac{21}{20}}
\Vert F\Vert_{S}\Vert G\Vert_{S}\Vert H\Vert_{S}.
\end{equation*}
We proceed by duality. Let $K\in L^2_{x,y}$, then by~\eqref{defN}
\begin{multline}
\<\sum_{\substack{A,B,C\\\max(A,B,C)\geq T^{\frac{1}{6}}}}\mathcal{N}^{t}[Q_{A}F,Q_{B}G,Q_CH],K\>_{L^2_{x,y}}=\\
=\sum_{\substack{A,B,C\\\max(A,B,C)\geq T^{\frac{1}{6}}}}   \int_{\R^{1+d}}e^{it \DD}(Q_{A} F)\overline{e^{it \DD}(Q_{B} G)}e^{it \DD}(Q_C H)\overline{e^{it \DD}K}dxdy.\label{sum}
\end{multline}
Next, by Sobolev ($\Vert K\Vert_{L^{6}_{x,y}}\leq C\Vert K\Vert_{H_{x,y}^{5/3}}$ since $d+1\leq 5$), one obtains
\begin{multline*}
\big\vert  \int_{\R^{1+d}}e^{it \DD}(Q_{A} F)\overline{e^{it \DD}(Q_{B} G)}e^{it \DD}(Q_C H)\overline{e^{it \DD}K}dxdy\big\vert \lesssim \\
\begin{aligned}
&\lesssim (ABC)^{-6-1/3} \Vert Q_{A} F\Vert_{H^{8}_{x,y}}\Vert Q_{B} G\Vert_{H^{8}_{x,y}}\Vert Q_C H\Vert_{H^{8}_{x,y}} \Vert K\Vert_{L^{2}_{x,y}},
\end{aligned}
\end{multline*}
and by summing up in~\eqref{sum} we get the result.

We now turn to~\eqref{est2}. Here we can follow the proof of~\cite[Lemma 3.2]{HPTV}. Denote by $\hbox{med}(A,B,C)$ the second largest dyadic number among $(A,B,C)$ and define the set $\Lambda$ of the     $(A,B,C)$ such that   $\hbox{med}(A,B,C)\leq T^\frac{1}{6}/16$ and $\max(A,B,C)\geq T^{\frac{1}{6}}$. The case when $(A, B, C) \notin \Lambda$ is treated exactly as in~\cite{HPTV}, so we don't reproduce the same argument. 

We consider a decomposition 
\begin{equation*} 
[T/4,2T]=\bigcup_{j\in J} I_j,\quad I_j=[jT^\frac{9}{10},(j+1)T^\frac{9}{10}]=[t_j,t_{j+1}],\quad \#J\lesssim T^{\frac{1}{10}}
\end{equation*}
and consider $\chi\in \mathcal{C}^\infty_c(\mathbb{R})$, $\chi\geq 0$ such that $\chi(x)=0$ if $\vert x\vert\geq 2$ and
\begin{equation*}
\sum_{k\in\mathbb{Z}}\chi(x-k)\equiv 1.
\end{equation*}
Then observe that the  left hand-side of~\eqref{est2} can be estimated by $C(E_{1}+E_{2})$, where
\begin{multline*}
E_{1}:=\Big\| \sum_{j\in J} \sum_{(A,B,C)\in \Lambda}\int_{T/2}^T\chi\big(\frac{t}{T^{\frac{9}{10}}}-j\big)\cdot\\
\cdot\Big(\mathcal{N}^{t}[Q_{A}F(t),Q_{B}G(t),Q_CH(t)]-\mathcal{N}^{t}[Q_{A}F(t_j),Q_{B}G(t_j),Q_CH(t_j)]\Big)dt\Big\|_{S}
\end{multline*}
and
\begin{equation*}
E_{2}:= \Big\|\sum_{j\in J} \sum_{(A,B,C)\in \Lambda}
\int_{T/2}^T\chi\big(\frac{t}{T^{\frac{9}{10}}}-j\big)\mathcal{N}^{t}[Q_{A}F(t_j),Q_{B}G(t_j),Q_CH(t_j)]dt\Big\|_{S}\,.
\end{equation*}
The term $E_{1}$ is estimated as in~\cite[Lemma 3.2]{HPTV}, and gives the expected contribution. For $E_{2}$ we write 
\begin{equation*}
E_{2} \leq \sum_{j\in J}\sum_{(A,B,C)\in \Lambda}E^{A,B,C}_{2,J},
\end{equation*}
with 
\begin{equation*}
 \Big\| \int_{T/2}^T\chi\big(\frac{t}{T^{\frac{9}{10}}}-j\big)\mathcal{N}^{t}[Q_{A}F^{a}(t_j),Q_{B}F^{b}(t_j),Q_CF^{c}(t_j)]dt\ \Big\|_{S}\,.
\end{equation*}
By Lemma~\ref{ControlSS+}, we only have to estimate 
\begin{equation*}
 \Big\| \int_{T/2}^T\chi\big(\frac{t}{T^{\frac{9}{10}}}-j\big)\mathcal{N}^{t}[Q_{A}F^{a}(t_j),Q_{B}F^{b}(t_j),Q_CF^{c}(t_j)]dt\ \Big\|_{L^{2}_{x,y}}\,,
\end{equation*}
and this will be done by   duality. Let $K\in L^2_{x,y}$, we consider
\begin{equation*}
\begin{split}
I_K&=\langle K, \int_{T/2}^T\chi\big(\frac{t}{T^{\frac{9}{10}}}-j\big)\mathcal{N}^{t}[Q_{A}F^{a}(t_j),Q_{B}F^{b}(t_j),Q_CF^{c}(t_j)]dt\rangle_{L^2_{x,y}\times L^2_{x,y}}\\
&=    \int_{T/2}^T  \int_{\R^{d+1}} \chi\big(\frac{t}{T^{\frac{9}{10}}}-j\big)e^{it \DD}(Q_{A} F^{a})\overline{e^{it \DD}(Q_{B} F^{b})}e^{it \DD}(Q_C F^{c})\overline{e^{it \DD}K}dxdydt,
\end{split}
\end{equation*}
where we may assume that $K=Q_DK$, $D\simeq \max(A,B,C)$. Using Lemma~\ref{ddBE}, we can estimate
\begin{equation*} 
I_K\lesssim D^{-1}
 \min_{\sigma\in\mathfrak{S}_3}\Vert Q_{\s(A)}F^{\sigma(a)}\Vert_{L^2_{x,y}}\Vert Q_{\s(B)}  F^{\sigma(B)}\Vert_{S}\Vert Q_{\s(C)}F^{\sigma(c)}\Vert_{S}.
 \end{equation*}
 
 The proof of~\eqref{est3} is similar.
\end{proof}

Let us recall the following result proved in~\cite{CKSTTSIMA}. We also reproduce the proof for the reader's convenience.
\begin{lemma}\label{1dBE}
Denote by $Q_{\lambda}$ the frequency localization in $x$ and assume that $\lambda\geq 10 \mu\geq 1$.
Then,  for all $f,g\in L^2(\R)$ we have the bound
\begin{equation}\label{BS0}
\big\Vert
\big(e^{it\partial_{xx}} Q_{\lambda}f \big)\big(e^{it\partial_{xx}}{Q_{\mu} g}\big)
\big\Vert_{L^2_{t,x}(\mathbb{R}\times\mathbb{R})}\lesssim \lambda^{-\frac{1}{2}}\Vert f\Vert_{L^2_x(\mathbb{R})}\Vert g\Vert_{L^2_x(\mathbb{R})}.
\end{equation}
\end{lemma}

\begin{proof}
Define $u(t)=e^{it\partial_{xx}}f $ and $v(t)=e^{it\partial_{xx}} g $ and assume that $Q_{\lambda}f=f$ and $Q_{\mu}g=g$. We proceed by duality. Let $F\in L^{2}(\R^{2})$, then by Parseval
\begin{equation*}
\<uv,F\>_{L^{2}_{t,x}}= \int_{\R^{2}} uv \ov{F}dxdt=\int_{\R^{2}}(\widehat{u}\star \widehat{v} )(t,\xi)\ov{\widehat{F}}(t,\xi)d\xi dt,
\end{equation*}
where $\wh{F}$ stands for the Fourier transform in $x$. Then if we denote by $\widetilde{F}$ the Fourier transform in $(t,x)$ we get with Cauchy-Schwarz
\begin{eqnarray*}
\<uv,F\>_{L^{2}_{t,x}}&=& \int_{\R^{3}} e^{-it(\xi^{2}_{1}+\xi^{2}_{2})} \wh{f}(\xi_{1}) \wh{g}(\xi_{2}) \ov{\widehat{F}}(t,\xi_{1}+\xi_{2})d\xi_{1}d\xi_{2} dt\\
&=&\int_{\R^{2}} \wh{f}(\xi_{1}) \wh{g}(\xi_{2}) \ov{\widetilde{{F}}}(-\xi^{2}_{1}-\xi^{2}_{2},\xi_{1}+\xi_{2})d\xi_{1}d\xi_{2}  \\
&\leq& \|f\|_{L^{2}(\R)} \|g\|_{L^{2}(\R)} \big( \int_{\substack{|\xi_{1}|\sim \lambda\\ {|\xi_{2}| \ll \lambda}}} |\ov{\widetilde{{F}}}(-\xi^{2}_{1}-\xi^{2}_{2},\xi_{1}+\xi_{2})|^{2}d\xi_{1}d\xi_{2} \big)^{1/2}.
\end{eqnarray*}
Now we make the change of variables $s=\xi_{1}+\xi_{2}$ and $r=-\xi^{2}_{1}-\xi^{2}_{2}$. The Jacobian of this transformation is $2|\xi_{1}-\xi_{2}|\sim \lambda$, hence
\begin{equation*}
\<uv,F\>_{L^{2}_{t,x}} \leq C \lambda^{-1/2}  \|f\|_{L^{2}(\R)} \|g\|_{L^{2}(\R)} \|\widetilde{F}\|_{L^{2}(\R^{2})},
\end{equation*}
which was the claim.
\end{proof}
\begin{lemma}\label{ddBE}
Denote by $Q_{\lambda}$ the frequency localization in $x$. Assume that $\lambda\geq 10 \mu\geq 1$ and that $u(t)=e^{it\DD}u_0$, $v(t)=e^{it\DD}v_0$. Then for $d\leq 4$ we have the bound
\begin{equation}\label{BS}
\Vert
Q_{\lambda}u{Q_{\mu} v}
\Vert_{L^2_{t,x,y}(\mathbb{R}\times\mathbb{R}^{1+d})}\lesssim \lambda^{-\frac{1}{2}}\min\big(\Vert u_0\Vert_{L^2_{x,y}(\mathbb{R}^{1+d})}\Vert v_0\Vert_{L^{2}_{x}\HH^{3}_{y}},\Vert v_0\Vert_{L^2_{x,y}(\mathbb{R}^{1+d})}\Vert u_0\Vert_{L^{2}_{x}\HH^{3}_{y}}\big).
\end{equation}
\end{lemma}

\begin{proof}
Let $A=\cup_{n\in \frac{1}{2}\Z}[n\pi-\frac{\pi}{8}, n\pi+ \frac{\pi}{8}]$. Then $\R=A\cup (A+\frac{\pi}{4})$. By time translation invariance and the unitarity of the flow on $L_{x, y}^2$ and $L^2_{x}\HH^{3}_y$, it is sufficient to prove~\eqref{BS} where the time integration is taken over $A$. For that, we will resort to the Lens transform to reduce to the corresponding problem with $\HH$ replaced by $-\Delta$, for which the result is almost trivial.

Let $f\in L^2(\R^d)$, and denote $v(t,\cdot)=e^{-itH}f$ and $u(t,\cdot)=e^{it\Delta}f$. The lens transform gives (see for instance~\cite{Carles2,Tao}) 
\begin{equation}\label{LensTransform}
u(t,x) =    \frac{1}{(1+4t^2)^{d/4}}    v \Big( \frac{\arctan(2t)}{2}  , \frac{x}{\sqrt{1+4t^2} } \Big)     e^{ \frac{i|x|^2t}{1+4t^2}  }.
\end{equation}
This means that for $t\in A$, we have that 
\begin{multline*}
\big(e^{it\partial_{xx}}e^{-it\HH} Q_{\lambda}\phi \big)\big(e^{it\partial_{xx}}e^{-it\HH} Q_{\lambda}\psi\big)\\
\quad=(1+\tan^22t)^{d/2}\big(e^{it\partial_{xx}}e^{i\frac{\tan 2t}{2}\Delta_{\R^{d}}}Q_{\lambda}\phi(x, \sqrt{1+\tan^2 2t} \,y)\big)\cdot\qquad\qquad\\
\cdot \big(e^{it\partial_{xx}}e^{i\frac{\tan 2t}{2}\Delta_{\R^{d}}}Q_{\mu}\psi(x, \sqrt{1+\tan^2 2t} \,y) \big)e^{ -iy^2\tan (2t)}.
\end{multline*}
Taking the $L^2(\R^{d}_y)$ norm in the previous equality, one gets:
\begin{align*}
&\big\|\big(e^{it\partial_{xx}}e^{-it\HH} Q_{\lambda}\phi \big)\big( e^{it\partial_{xx}}e^{-it\HH} Q_{\lambda}\psi\big)\big\|_{L^2(\R_y^{d})}\\
&\quad=(1+\tan^22t)^{d/4}\big\|\big( e^{it\partial_{xx}}e^{i\frac{\tan 2t}{2}\Delta_{\R^{d}}}Q_{\lambda}\phi(x, y)\big) \big(e^{it\partial_{xx}}e^{i\frac{\tan 2t}{2}\Delta_{\R^{d}}}Q_{\mu}\psi(x, y)\big)\big\|_{L^2(\R_y^{d})}\\
&\quad\leq 2^{d/4}\big\|\big( e^{it\partial_{xx}}e^{i\frac{\tan 2t}{2}\Delta_{\R^{d}}}Q_{\lambda}\phi(x, y)\big) \big(e^{it\partial_{xx}}e^{i\frac{\tan 2t}{2}\Delta_{\R^{d}}}Q_{\mu}\psi(x, y)\big)\big\|_{L^2(\R_y^{d})}
\end{align*}
where we used in the last inequality that $t\in A$. But, by Plancherel's theorem
\begin{align*}
&\big\|\big( e^{it\partial_{xx}}e^{i\frac{\tan 2t}{2}\Delta_{\R^{d}}}Q_{\lambda}\phi(x, y)\big) \big(e^{it\partial_{xx}}e^{i\frac{\tan 2t}{2}\Delta_{\R^{d}}}Q_{\mu}\psi(x, y)\big)\big\|_{L^2(\R_y^{d})}\\
&\quad=\left\|\int_{\R^{d}_{\eta_{1}}}e^{-i\frac{\tan 2t}{2}(|\eta_{1}|^2+|\eta-\eta_{1}|^2)}e^{it\partial_{xx}}\mathcal F_{y}Q_\lambda \phi(x, \eta_{1})e^{it\partial_{xx}}\mathcal F_{y}Q_\mu\psi(x, \eta-\eta_{1})d\eta_{1}\right\|_{L^2(\R^{d}_{\eta})}\\
&\quad\leq \left\|\int_{\R^{d}_{\eta_{1}}}\left| e^{it\partial_{xx}}\mathcal F_{y}Q_\lambda \phi(x, \eta_{1})e^{it\partial_{xx}}\mathcal F_{y}Q_\mu\psi(x, \eta-\eta_{1})\right|d\eta_{1}\right\|_{L^2(\R^{d}_{\eta})}.
\end{align*}
Taking $L^{2}_{t,x}(A\times \R)$ and using Minkowski's inequality and the 1D Strichartz estimate~\eqref{BS0} we arrive at
\begin{align*}
&\big\|\big( e^{it\partial_{xx}}e^{i\frac{\tan 2t}{2}\Delta_{\R^{d}}}Q_{\lambda}\phi(x, y)\big) \big(e^{it\partial_{xx}}e^{i\frac{\tan 2t}{2}\Delta_{\R^{d}}}Q_{\mu}\psi(x, y)\big)\big\|_{L^2_{t,x, y}(A\times \R\times \R^{d})}\\
&\quad\leq \lambda^{-1/2}\left\|\int_{\R^{d}_{\eta_{1}}}\|\mathcal F_{y} \phi(x, \eta_{1})\|_{L^2_x(\R)}\|\mathcal F_{y} \psi(x, \eta-\eta_{1})\|_{L^2_x(\R)}d\eta_{1}\right\|_{L^2(\R^{d}_{\eta})}\\
&\quad\lesssim \lambda^{-1/2}\|\phi\|_{L^2_{x,y}}\|\psi\|_{L_x^2 H_y^{\frac{d}{2}+1}}\lesssim \lambda^{-1/2}\|\phi\|_{L^2_{x,y}}\|\psi\|_{L_x^2 \HH_y^{\frac{d}{2}+1}},
\end{align*}
which is the needed the result. Note that in the last inequality we used the fact that the $\HH^{d/2+1}$ norm controls the Euclidean $H^{d/2+1}$ norm.
\end{proof}

Recall that $F=\sum_{q\geq0}F_{q}$ is the decomposition of $F$ according to the eigenspaces of the harmonic oscillator. Then given any trilinear operator $\mathfrak I$ we define
\begin{equation*}
P_{\omega=0} \mathfrak J[F, G, H]:=\sum_{(p,q,r,s)\in \Gamma_0} \Pi_p \mathfrak  I[F_q, G_r, H_s], 
\end{equation*}
and
\begin{equation*}
P_{\omega\neq 0} \mathfrak I[F, G, H]:=\sum_{\omega\neq 0}\sum_{(p,q,r,s)\in \Gamma_\omega} \Pi_p \mathfrak I[F_q, G_r, H_s].
\end{equation*}
We then  split the nonlinearity as follows:
\begin{equation*}
 \mathcal{N}^{t}[F, G, H]=\mathcal N_0^{t}[F, G, H]+ \widetilde{\mathcal N^{t}}[F, G, H],
\end{equation*}
with 
\begin{equation}\label{defN0}
\mathcal{N}_0^{t}[F, G, H]:=P_{\omega=0} \mathcal{N}^{t}[F, G, H], \quad \widetilde{\mathcal N^{t}}[F, G, H]:=P_{\omega\neq 0} \mathcal{N}^{t}[F, G, H].
\end{equation}

\subsection{The fast oscillations: study of \texorpdfstring{$\widetilde{\mathcal{N}^{t}}$}{N}}\label{FastOsSec}
The main purpose of this subsection is to prove the following:
\begin{lemma}\label{FastOsLem}
Let $1\leq T\leq T^\ast$. Assume that $F$, $G$, $H$: $\mathbb{R} \to S$ satisfy~\eqref{leq1} and
\begin{equation*}
F=Q_{\leq T^{1/6}}F, \quad G=Q_{\leq T^{1/6}}G, \quad H=Q_{\leq T^{1/6}}H\,.
\end{equation*}
Then we can write
$$
\widetilde{\mathcal{N}^{t}}[F(t),G(t),H(t)]=\widetilde{\mathcal{E}_{1}^{t}}[F(t),G(t),H(t)]+\mathcal{E}_{2}^{t}[F(t),G(t),H(t)],
$$
and if we set $\widetilde{\mathcal{E}_{1}}(t):=\widetilde{\mathcal{E}^{t}_{1}}[F(t),G(t),H(t)]$
and $\mathcal{E}_{2}(t):=\mathcal{E}_{2}^{t}[F(t),G(t),H(t)]$ then it holds that, uniformly in $1\leq T\leq T^\ast$,
\begin{equation*}
T^{1+2\delta}\sup_{T/4\leq t\leq T^\ast}\Vert \widetilde{\mathcal{E}_{1}}(t)\Vert_{S}\lesssim 1,\quad
T^{1/10}\sup_{T/4\leq t\leq T^\ast}\Vert \mathcal{E}_3(t)\Vert_{S}\lesssim 1,
\end{equation*}
where $\mathcal{E}_{2}(t)=\partial_t\mathcal{E}_3(t)$.
Assuming in addition that~\eqref{BA+} holds we have 
\begin{equation*}
T^{1+2\delta}\sup_{T/4\leq t\leq T^\ast}\Vert \widetilde{\mathcal{E}_{1}}(t)\Vert_{S^+}\lesssim 1,
\qquad T^{1/10}\sup_{T/4\leq t\leq T^\ast}\Vert \mathcal{E}_3(t)\Vert_{S^+}\lesssim 1.
\end{equation*}
\end{lemma}

The proof of this lemma follows the so-called ``space-time" resonance approach~\cite{GeMaSh}.
\begin{proof} 
Set 
\begin{equation*}
 F^{a}=Q_{\leq T^\frac{1}{6}}F^{a},\qquad  F^{b}=Q_{\leq T^\frac{1}{6}}F^{b}, \qquad  F^{c}=Q_{\leq T^\frac{1}{6}}F^{c}.
\end{equation*}
Let us decompose  $\widetilde{\mathcal N^{t}}$ along the non-resonant level sets as follows:
\begin{equation*}
\mathcal{F}_{x}\widetilde{\mathcal{N}^{t}}[F^{a},F^{b},F^{c}](\xi,y)=
P_{\om \neq 0}\big(\mathcal{O}^{t}_{1}[F^{a},F^{b},F^{c}](\xi,y)+\mathcal{O}^{t}_{2}[F^{a},F^{b},F^{c}](\xi,y)\big), 
\end{equation*}
where $P_{\om\neq 0}$ is the projector on the non resonant frequencies, and
\begin{multline}\label{defO1}
\mathcal{O}^{t}_{1}[F^{a},F^{b},F^{c}](\xi,y):=\\
e^{-it\HH_{d}}\int_{\mathbb{R}^2}e^{2it\eta\kappa}(1-\varphi(t^{\frac{1}{4}}\eta\kappa))\widehat{e^{it\HH_{d}} F^{a}}(\xi-\eta)\overline{\widehat{e^{it\HH_{d}} F^{b}}}(\xi-\eta-\kappa)\widehat{e^{it\HH_{d}} F^{c}}(\xi-\kappa)d\eta d\kappa,
\end{multline}
\begin{multline*}
\mathcal{O}^{t}_{2}[F^{a},F^{b},F^{c}](\xi,y):=\\
e^{-it\HH_{d}}\int_{\mathbb{R}^2}e^{2it\eta\kappa}\varphi(t^{\frac{1}{4}}\eta\kappa)\widehat{e^{it\HH_{d}} F^{a}}(\xi-\eta)\overline{\widehat{e^{it\HH_{d}} F^{b}}}(\xi-\eta-\kappa)\widehat{e^{it\HH_{d}} F^{c}}(\xi-\kappa)d\eta d\kappa.
\end{multline*}

$\bullet$ {\bf Estimation of $\mathcal{O}^{t}_{2}$:}  Denote by 
 \begin{equation*}
 \mathcal{O}^{t}_{2,\om}[F^{a},F^{b},F^{c}]:=\sum_{(p,q,r,s)\in \Gamma_{\om}}\Pi_{p} \int_{\mathbb{R}^2}e^{2it\eta\kappa}\varphi(t^{\frac{1}{4}}\eta\kappa)\widehat{ F^{a}_{q}}(\xi-\eta)\overline{\widehat{ F^{b}_{r}}}(\xi-\eta-\kappa)\widehat{F^{c}_{s}}(\xi-\kappa)d\eta d\kappa,
 \end{equation*}
 then we can write
 \begin{equation}\label{decomp}
 P_{\om\neq 0} \mathcal{O}^{t}_{2}[F^{a},F^{b},F^{c}] =\sum_{\om\neq 0}e^{-i\om t} \mathcal{O}^{t}_{2,\om}[F^{a},F^{b},F^{c}].
 \end{equation}
Next observe that  we have the following relation
\begin{multline}\label{TIPP}
 e^{-it\omega}\mathcal{O}^{t}_{2,\om}[F^{a},F^{b},F^{c}]=\partial_t\Big( \frac{e^{-it\omega}-1}{-i\omega}\mathcal{O}^{t}_{2,\om}[F^{a},F^{b},F^{c}]\Big)
+\frac{e^{-it\omega}-1}{i\omega}\left(\partial_t\mathcal{O}^{t}_{2,\om}\right)[F^{a},F^{b},F^{c}]
\\
+\frac{e^{-it\omega}-1}{i\omega}\Big(\mathcal{O}^{t}_{2,\om}[\partial_tF^{a},F^{b},F^{c}]+
 \mathcal{O}^{t}_{2,\om}[F^{a},\partial_{t}F^{b},F^{c}]+\mathcal{O}^{t}_{2,\om}[F^{a},F^{b},\partial_{t}F^{c}]\Big),
\end{multline}
where
 \begin{equation*}
\left(\partial_t\mathcal{O}^{t}_{2,\om}\right)[F^{a},F^{b},F^{c}](\xi):=\sum_{(p,q,r,s)\in \Gamma_{\om}}\Pi_{p} \int_{\mathbb{R}^2}\partial_{t}\big(e^{2it\eta\kappa}\varphi(t^{\frac{1}{4}}\eta\kappa)\big)\widehat{ F^{a}_{q}}(\xi-\eta)\overline{\widehat{ F^{b}_{r}}}(\xi-\eta-\kappa)\widehat{F^{c}_{s}}(\xi-\kappa)d\eta d\kappa.
 \end{equation*}
Now set
\begin{multline}\label{defO2}
\mathcal{O}^{t,\tau}_{2}[F^{a},F^{b},F^{c}](\xi) :=\\
e^{-i\tau\HH_{d}}\int_{\mathbb{R}^2}e^{2it\eta\kappa}\varphi(t^{\frac{1}{4}}\eta\kappa)\widehat{e^{i\tau\HH_{d}} F^{a}}(\xi-\eta)\overline{\widehat{e^{i\tau\HH_{d}} F^{b}}}(\xi-\eta-\kappa)\widehat{e^{i\tau\HH_{d}} F^{c}}(\xi-\kappa)d\eta d\kappa 
\end{multline}
and
\begin{multline*}
\big(\partial_{t}\mathcal{O}^{t,\tau}_{2}\big)[F^{a},F^{b},F^{c}](\xi): =\\
e^{-i\tau\HH_{d}}\int_{\mathbb{R}^2} \partial_t\Big(e^{2it\eta\kappa}\varphi(t^{\frac{1}{4}}\eta\kappa)\Big)\widehat{e^{i\tau\HH_{d}} F^{a}}(\xi-\eta)\overline{\widehat{e^{i\tau\HH_{d}} F^{b}}}(\xi-\eta-\kappa)\widehat{e^{i\tau\HH_{d}} F^{c}}(\xi-\kappa)d\eta d\kappa .
\end{multline*}

The reason for introducing those operators is the fact that for all $t,\tau$ and $\omega$
\begin{multline*}
e^{-i\omega \tau}\mathcal{O}^{t}_{2,\omega}[F^{a},F^{b},F^{c}](\xi)=\\
\begin{aligned}
&=\sum_{(p,q,r,s)\in \Gamma_{\om}}\Pi_{p} e^{-i\tau \HH_{d}}\int_{\mathbb{R}^2}e^{2it\eta\kappa}\varphi(t^{\frac{1}{4}}\eta\kappa)\widehat{ e^{i\tau \HH_{d}}F^{a}_{q}}(\xi-\eta)\overline{\widehat{ e^{i\tau \HH_{d}}F^{b}_{r}}}(\xi-\eta-\kappa)\widehat{e^{i\tau \HH_{d}}F^{c}_{s}}(\xi-\kappa)d\eta d\kappa\\
&=\sum_{(p,q,r,s)\in \Gamma_{\om}}\Pi_{p} \mathcal{O}^{t, \tau}_{2}[F^{a}_q, F^{b}_r, F^{c}_s](\xi)
\end{aligned}
\end{multline*}
so that 
$$
\sum_{\omega\neq 0}e^{-i\omega \tau}\mathcal{O}^{t}_{2,\omega}[F^{a}, F^{b}, F^{c}]=P_{\omega\neq 0} \mathcal{O}^{t, \tau}_{2}[F^{a}, F^{b}, F^{c}].
$$

Denote by $[\tau]$ the integer part of the real number $\tau$, then   by~\eqref{decomp} and~\eqref{TIPP}
\begin{multline}\label{tipp2}
 P_{\om\neq 0} \mathcal{O}^{t}_{2}[F^{a},F^{b},F^{c}]=-\partial_t\big(  \mathcal{F}_{x}\mathcal{E}_3\big)
- P_{\om\neq 0}   \int_{2\pi[\frac{t}{2\pi}]}^{t}  \big(\partial_{t}\mathcal{O}^{t,\tau}_{2} \big) [F^{a},F^{b},F^{c}]d\tau
\\
-P_{\om\neq 0}   \int_{2\pi[\frac{t}{2\pi}]}^{t}  \big(\mathcal{O}^{t,\tau}_{2} [\partial_tF^{a},F^{b},F^{c}]+\mathcal{O}^{t,\tau}_{2} [F^{a},\partial_tF^{b},F^{c}]+\mathcal{O}^{t,\tau}_{2} [F^{a},F^{b},\partial_tF^{c}]  \big)      d\tau,   
\end{multline}
where $\mathcal{E}_3$ is defined by 
\begin{eqnarray}\label{defe3}
 \mathcal{F}_{x}\mathcal{E}_3(\xi,y)&:=&\sum_{\omega\neq 0}\sum_{(p,q,r,s)\in\Gamma_\omega}\frac{1-e^{-it\omega}}{i\omega}\mathcal{O}^{t}_{2,\om}[F^{a}_q,F^{b}_r,F^{c}_s](\xi,y)\nonumber\\
 &=& \int_{2\pi[\frac{t}{2\pi}]}^{t} \sum_{\omega\neq 0}\sum_{(p,q,r,s)\in\Gamma_\omega}e^{-i\tau \omega}\mathcal{O}^{t}_{2,\om}[F^{a}_q,F^{b}_r,F^{c}_s](\xi,y)d \tau\nonumber\\
&=& \int_{2\pi[\frac{t}{2\pi}]}^{t}  P_{\om\neq 0} \mathcal{O}^{t,\tau}_{2}[F^{a},F^{b},F^{c}] (\xi,y)d\tau.
\end{eqnarray}
We now estimate the contribution of each term in~\eqref{tipp2}.\medskip

$\star$ We first consider the term $ \mathcal{F}_{x}\mathcal{E}_3$.  We can define the multliplier $m$ which appears in the definition of  $\mathcal{O}^{t}_{2}$ by 
\begin{equation*}
m(\eta, \kappa):=\varphi(t^{1/4}\eta \kappa) \varphi((10T)^{-1/6}\eta)\varphi((10T)^{-1/6}\kappa).
\end{equation*}
It's not hard to see (cf.~\cite[Remark 3.5]{HPTV}) that the following bound holds $\|\mathcal F_{\eta\kappa} \widetilde{m}\|_{L^1(\mathbb{R}^2)}\lesssim t^{\frac{\delta}{100}}$. We now apply  Lemma~\ref{CM} below with $(p,q,r,s)=(2,2,\infty,\infty)$, $f=e^{-it \partial^{2}_{x}}e^{is\HH_{d}} F^{\sigma(a)},\dots$ and we get for $t\geq T/4$ and $\tau \in \big[2\pi[\frac{t}{2\pi}], t\big]$
\begin{multline*} 
\Vert \mathcal{O}^{t,\tau}_{2} [F^{a},F^{b},F^{c}]\Vert_{L^2_{\xi}}=\Vert \mathcal{F}_{\xi}\mathcal{O}^{t,\tau}_{2} [F^{a},F^{b},F^{c}]\Vert_{L^2_{x}}\\
\begin{aligned}
&\lesssim(1+|t|)^\frac{\delta}{100}
\min_{\sigma\in\mathfrak{S}_3}\Vert e^{-it \partial^{2}_{x}}e^{is\HH_{d}} F^{\sigma(a)}\Vert_{L^2_{x}}\Vert e^{-it \partial^{2}_{x}}e^{is\HH_{d}}F^{\sigma(b)}\Vert_{L^{\infty}_{x}}\Vert e^{-it \partial^{2}_{x}}e^{is\HH_{d}}F^{\sigma (c)}\Vert_{L^{\infty}_{x}}.
\end{aligned}
\end{multline*}
Then we take the $L^{2}_{y}$-norm and by~\eqref{estdisp} we get 
\begin{equation}\label{o22}
\Vert \mathcal{O}^{t,\tau}_{2} [F^{a},F^{b},F^{c}]\Vert_{L^2_{\xi,y}}\lesssim (1+\vert t\vert)^{-1+\frac{\delta}{100}}\min_{\sigma\in\mathfrak{S}_3}\Vert   F^{\sigma(a)}\Vert_{L^2_{x,y}}\Vert F^{\sigma(b)}\Vert_{S}\Vert F^{\sigma(c)}\Vert_{S}.
\end{equation}

Next use that $\big|    t- 2\pi[\frac{t}{2\pi}]\big|\leq C $ to get 
\begin{equation}\label{calc}
\big\| \mathcal{F}_{x}\mathcal{E}_3\big\|_{L^{2}_{\xi,y}}\lesssim (1+\vert t\vert)^{-1+\frac{\delta}{100}}\min_{\sigma\in\mathfrak{S}_3}\Vert   F^{\sigma(a)}\Vert_{L^2_{x,y}}\Vert F^{\sigma(b)}\Vert_{S}\Vert F^{\sigma(c)}\Vert_{S},
\end{equation}
\medskip
and the estimates in $S$ and $S^{+}$ norms follow from Lemma~\ref{ControlSS+}.

$\star$ Since $(1+\vert t\vert)^{1/4}(\partial_t\mathcal{O}_{2}^{t,\tau})$ satisfies similar estimates as $\mathcal{O}_{2}^{t,\tau}$, the second term in the right hand-side of~\eqref{tipp2} is acceptable.\medskip

$\star$ The contribution of the terms in the second line of~\eqref{tipp2} is estimated as in~\eqref{o22}, and by using the definition of the $X_{T^\ast}$ norm.\medskip

This ends the estimation of $\mathcal{O}^{t}_{2}$.\medskip

$\bullet$ {\bf Estimation of $\mathcal{O}^{t}_{1}$:}  Here we show that for $|t|\geq T/4$
\begin{equation*} 
\Vert \mathcal{O}^{t}_{1}[F^{a},F^{b},F^{c}]\Vert_{L^2_{\xi,y}}\lesssim  T^{-\frac{201}{200}}\min_{\sigma\in\mathfrak{S}_3}\Vert F^{\sigma(a)}\Vert_{L^2_{x,y}}\Vert F^{\sigma(b)}\Vert_{S}\Vert F^{\sigma(c)}\Vert_{S}.
\end{equation*}
It is enough to prove that
\begin{equation}\label{OSmall1}
\begin{split}
\Vert \mathcal{O}^{t}_{1}[F,G,H]\Vert_{L^2_{\xi,y}}&\lesssim  T^{-\frac{201}{200}}\Vert F\Vert_{L^2_{x,y}}\Vert G\Vert_{S}\Vert H\Vert_{S},
\end{split}
\end{equation}
since the other inequalities  follow by symmetry and conjugation.

Then, observe that by $\mathcal{O}_{1}^{t}+\mathcal{O}_{2}^{t}$ satisfies the estimate~\eqref{SuffLem} (Lemma~\ref{CM}), and therefore we also have the bound
\begin{equation}\label{o1}
\Vert \mathcal{O}^{t}_{1} [F,G,H]\Vert_{L^2_{\xi,y}}   \lesssim(1+|t|)^\frac{\delta}{100}
 \Vert e^{it\DD}  F\Vert_{L^2_{x,y}}\Vert e^{it\DD}G\Vert_{L^{\infty}_{x,y}}\Vert e^{it\DD}H\Vert_{L^{\infty}_{x,y}}.
\end{equation}

We first decompose
\begin{equation*}
\begin{split}
G=G_c+G_f,\quad H=H_c+H_f,\quad G_c(x)=\varphi(\frac{x}{T^{\frac{3}{4}}})G(x),\quad H_c(x)=\varphi(\frac{x}{T^{\frac{3}{4}}})H(x).
 \end{split}
\end{equation*}
Arguing as in~\eqref{estFt} but using instead the inequality~\eqref{L11} in the general case, one obtains that 
\begin{equation}\label{418}
  \big\Vert e^{it\DD}F\big\Vert_{L^{\infty}_{x,y}}  \lesssim   \vert t\vert^{-\frac{1}{2}}    \big\Vert F\big\Vert^{1-\frac1{2\a}}_{L^{2}_{x}\HH_{y}^{N}}    \big\Vert |x|^{\a}F\big\Vert^{\frac1{2\a}}_{L^{2}_{x,y}},\quad \text{for all}\quad \frac{1}{2}<\alpha\leq 1. 
\end{equation}
We apply this estimate to $F$ and $xF$ and get that 
\begin{equation}\label{pinguin}
\begin{split}
\big\Vert e^{it\DD}F\big\Vert_{L^{\infty}_{x,y}}  \lesssim&  \;R^{-\beta} \vert t\vert^{-\frac{1}{2}}    \big\Vert F\big\Vert_{S}, \quad\text{if $F$ is supported on $|x|\geq R$}\\[3pt]
 \big\Vert e^{it\DD}xF\big\Vert_{L^{\infty}_{x,y}}  \lesssim& \; R^{1-\gamma}|t|^{-\frac{1}{2}} \big\Vert F\big\Vert_{S}, \quad\text{if $F$ is supported on $|x|\leq R$}
\end{split}
\end{equation}
 for all $\beta, \gamma<1/2$. Fixing $\beta=\gamma=\frac{1}{3}$ for concreteness, one obtains from~\eqref{o1} and the first inequality in~\eqref{pinguin} that~\eqref{OSmall1} is a consequence of the estimate
\begin{equation*}
\Vert \mathcal{O}^{t}_{1}[F,G_c,H_c]\Vert_{L^2_{\xi,y}}\lesssim T^{-\frac{101}{100}}\Vert F\Vert_{L^2_{x,y}}\Vert G\Vert_{S}\Vert H\Vert_{S}.
\end{equation*}

To prove this estimate, we integrate by parts in the $\kappa$ integral in~\eqref{defO1} to obtain:
\begin{align*}
&\mathcal{O}^{t}_{1}[F,G,H](\xi,y)=\\&(2it)^{-1}
e^{-it\HH_{d}}\int_{\mathbb{R}^2}e^{2it\eta\kappa}\eta^{-1}\widehat{e^{it\HH_{d}} F}(\xi-\eta)\partial_{\kappa}\left[(1-\varphi(t^{\frac{1}{4}}\eta\kappa))\overline{\widehat{e^{it\HH_{d}} G}}(\xi-\eta-\kappa)\widehat{e^{it\HH_{d}} H}(\xi-\kappa)\right]d\eta d\kappa.
\end{align*}
The contribution of the term when the $\kappa$ derivative falls on the multiplier is easy to bound. So, suppose that the $\kappa$ derivative falls on $G$. In this case, we note that on the support of integration, we necessarily have $\vert \eta\vert\gtrsim t^{-1/4}|\kappa|^{-1}\gtrsim T^{-\frac{5}{12}}$. Thanks to~\eqref{o1}, this allows to bound the corresponding contribution as:
\begin{multline*}
CT^{-\frac{7}{12}}
\Big\|\int_{\mathbb{R}^2}e^{2it\eta\kappa}\frac{(1-\varphi(T^{\frac{5}{12}}\eta))}{T^{\frac{5}{12}}\eta}(1-\varphi(t^{\frac{1}{4}}\eta\kappa))\cdot\\
\qquad\qquad\qquad\qquad\cdot\widehat{e^{it\HH_{d}} F}(\xi-\eta)\overline{\widehat{e^{it\HH_{d}} xG}}(\xi-\eta-\kappa)\widehat{e^{it\HH_{d}} H}(\xi-\kappa)d\eta d\kappa \Big\|_{L^2_{\xi, y}}\lesssim\\
\begin{aligned}
&\lesssim T^{-\frac{7}{12}+\frac{\delta}{100}}\|F\|_{L^2_{x,y}}\|e^{it\DD} xG\|_{L^\infty_{x,y}}\|e^{it\DD} H\|_{L^\infty_{x,y}}\\
& \lesssim T^{-\frac{13}{12}+\frac{\delta}{100}}\|F\|_{L^2_{x,y}}\|G\|_{S}\|H\|_{S},
\end{aligned}
\end{multline*}
where the last estimate follows from~\eqref{418} and the second line in~\eqref{pinguin}.

The case when the $\kappa$ derivative falls on $H$ is similar. This finishes the estimation of $\mathcal O^{t}_{1}$ in~\eqref{OSmall1}.\medskip

$\bullet$ {\bf Conclusion of the proof of Lemma~\ref{FastOsLem}:} Define $\mathcal{E}_{2}=\partial_{t}\mathcal{E}_{3}$, where $\mathcal{E}_{3}$ is given by~\eqref{defe3}. Then by~\eqref{calc} it satisfies the conclusion of Lemma~\ref{FastOsLem}. The term $\widetilde{\mathcal{E}_{1}}$ is then  defined as the sum of~$\mathcal{O}^{t}_{1}$ and the remaining terms in~\eqref{tipp2}, and by the previous estimates it satisfies the claim.
\end{proof}
 
\subsection{The resonant part: study of \texorpdfstring{$\mathcal{N}_{0}^{t}$}{N0}  }
The main contribution to $\mathcal N^{t}$ comes from the resonant part of the nonlinearity $\mathcal{N}^{t}_0$ which was defined in~\eqref{defN0} as
\begin{equation*}
\mathcal{N}^{t}_0[F, G, H]=\sum_{(p,q,r,s)\in \Gamma_0} \Pi_p \mathcal{N}^{t}[F_q, G_r, H_s].
\end{equation*}
It can also    be written 
\begin{equation}\label{defPi}
\mathcal{N}^{t}_0[F,G,H]= \frac1{\pi} \int_{s=0}^{\pi}  e^{it\partial^{2}_{x}-is\HH_{d}}\Big( e^{-it \partial^{2}_{x}+is\HH_{d}}F\cdot e^{it\partial^{2}_{x}-is\HH_{d}}\overline{G}\cdot e^{-it\partial^{2}_{x}+is\HH_{d}}H\Big)ds,
\end{equation}
and this latter formulation will be convenient to exploit Strichartz estimates.\ligne

We have the following result

\begin{lemma}\label{lem46} Assume that $N\geq7$, let $t\geq1$ and denote by 
\begin{equation*}
\Vert F\Vert_{\tilde{Z}_t}:=\Vert F\Vert_{Z}+(1+\vert t\vert)^{-\delta}\Vert F\Vert_{S},                                                    
\end{equation*}
then  
\begin{align}
\|\mathcal{N}^{t}_0[F^{a}, F^{b}, F^{c}]\|_{S}\lesssim& (1+|t|)^{-1}\max_{\s\in\mathfrak{S}_3}\Vert F^{\sigma(a)}\Vert_{\tilde{Z}_t}  \Vert F^{\sigma(b)}\Vert_{\tilde{Z}_t} \Vert F^{\sigma(c)}\Vert_{S}\label{est4}
\end{align}
and
\begin{equation}\label{est41}
\begin{split}
\Vert \mathcal{N}^{t}_0[F^{a},F^{b},F^{c}]\Vert_{S^+}\lesssim& (1+\vert t\vert)^{-1}\max_{\s\in\mathfrak{S}_3}\Vert F^{\sigma(a)}\Vert_{\tilde{Z}_t}  \Vert F^{\sigma(b)}\Vert_{\tilde{Z}_t} \Vert F^{\sigma(c)}\Vert_{S^+}\\
&\quad+(1+\vert t\vert)^{-1+2\delta}\max_{\s\in\mathfrak{S}_3}\Vert F^{\sigma(a)}\Vert_{\tilde{Z}_t}  \Vert F^{\sigma(b)}\Vert_{S}  \Vert F^{\sigma(c)}\Vert_{S}.
\end{split}
\end{equation}
Moreover, we have that 
\begin{equation}\label{di1}
\big\| \mathcal N_0^{t}[F,G, H]-\frac{\pi}{t}\mathcal R[F, G, H]\big\|_{Z}\lesssim (1+|t|)^{-1-10\delta}\|F\|_{S}\|G\|_{S}\|H\|_{S},
\end{equation}
\begin{equation}\label{differNRS}
\big\| \mathcal N_0^{t}[F,G, H]-\frac{\pi}{t}\mathcal R[F, G, H]\big\|_{S}\lesssim (1+|t|)^{-1-20\delta}\|F\|_{S^+}\|G\|_{S^+}\|H\|_{S^+}.
\end{equation}
Furthermore
 \begin{align}
\|\mathcal{R}[F^{a}, F^{b}, F^{c}]\|_{S}\lesssim&  \max_{\s\in\mathfrak{S}_3}\Vert F^{\sigma(a)}\Vert_{\tilde{Z}_t}  \Vert F^{\sigma(b)}\Vert_{\tilde{Z}_t} \Vert F^{\sigma(c)}\Vert_{S}\label{esti1}
\end{align}
and
\begin{equation}\label{esti2}
\begin{split}
\Vert \mathcal{R}[F^{a},F^{b},F^{c}]\Vert_{S^+}\lesssim& \max_{\s\in\mathfrak{S}_3}\Vert F^{\sigma(a)}\Vert_{\tilde{Z}_t}  \Vert F^{\sigma(b)}\Vert_{\tilde{Z}_t} \Vert F^{\sigma(c)}\Vert_{S^+}\\
&\quad+(1+\vert t\vert)^{2\delta}\max_{\s\in\mathfrak{S}_3}\Vert F^{\sigma(a)}\Vert_{\tilde{Z}_t}  \Vert F^{\sigma(b)}\Vert_{S}  \Vert F^{\sigma(c)}\Vert_{S}.
\end{split}
\end{equation}
 \end{lemma}
 
\begin{proof} $\bullet$ We  start with~\eqref{est4}. By  Lemma~\ref{ControlSS+}, it is sufficient to prove that 
\begin{equation}\label{N0L2}
\Vert \mathcal{N}^{t}_0[F^{a},F^{b},F^{c}]\Vert_{L^2_{x,y}}\lesssim (1+\vert t\vert)^{-1}\min_{\sigma\in\mathfrak{S}_3}\Vert F^{\sigma(a)}\Vert_{L^2_{x,y}}\Vert F^{\sigma(b)}\Vert_{\tilde{Z}_t}\Vert F^{\sigma(c)}\Vert_{\tilde{Z}_t},
\end{equation}
and by symmetry, we only consider the case $\s=id$. Let $G\in L^2_{x,y}$, and denote by $u^{j}(t,x)= e^{-it \partial^{2}_{x} }F^{j}$ then by~\eqref{defPi}
\begin{equation*}
\<    \mathcal{N}^{t}_0[F^{a},F^{b},F^{c}],G     \>_{L^2_{x,y}}=
\frac1{\pi} \int_{x\in \R} \int_{s=0}^{\pi} \int_{y\in \R^{d}} \Big( e^{ is\HH_{d}}u^{a}\cdot e^{-is\HH_{d}}\overline{u^{b}}\cdot e^{ is\HH_{d}}u^{c}\Big) e^{it\partial^{2}_{x}-is\HH_{d}} \ov{G} dy ds dx   .
\end{equation*}
By Lemma~\ref{Lem32}
\begin{multline*}
\Big\vert  \int_{s=0}^{\pi} \int_{y\in \R^{d}} \Big( e^{ is\HH_{d}}u^{a}\cdot e^{-is\HH_{d}}\overline{u^{b}}\cdot e^{ is\HH_{d}}u^{c}\Big) e^{it\partial^{2}_{x}-is\HH_{d}} \ov{G} dy ds\Big\vert  \leq \\
\leq C \|e^{-it\partial^{2}_{x}}G\|_{L_{y}^{2}}\|e^{-it\partial^{2}_{x}}F^{a}\|_{L_{y}^{2}}\|e^{-it\partial^{2}_{x}}F^{b}\|_{\HH^{1}_{y}}\|e^{-it\partial^{2}_{x}}F^{c}\|_{\HH^{1}_{y}},
\end{multline*}
which in turn implies by Cauchy-Schwarz
\begin{eqnarray*}
\big\vert \<    \mathcal{N}^{t}_0[F^{a},F^{b},F^{c}],G     \>_{L^2_{x,y}}\big\vert &\leq& C  \|e^{-it\partial^{2}_{x}}G\|_{L^{2}_{x}L_{y}^{2}}\|e^{-it\partial^{2}_{x}}F^{a}\|_{L^{2}_{x}L_{y}^{2}}\|e^{-it\partial^{2}_{x}}F^{b}\|_{L^{\infty}_{x}\HH^{1}_{y}}\|e^{-it\partial^{2}_{x}}F^{c}\|_{L^{\infty}_{x}\HH^{1}_{y}}\nonumber \\
&=& C  \| G\|_{L^{2}_{x}L_{y}^{2}}\| F^{a}\|_{L^{2}_{x}L_{y}^{2}}\|e^{-it\partial^{2}_{x}}F^{b}\|_{L^{\infty}_{x}\HH^{1}_{y}}\|e^{-it\partial^{2}_{x}}F^{c}\|_{L^{\infty}_{x}\HH^{1}_{y}}.
\end{eqnarray*}
We claim that 
\begin{equation}\label{est7}
\|e^{-it\partial^{2}_{x}}F\|_{L^{\infty}_{x}\HH^{1}_{y}}\leq C(1+|t|)^{-1/2}\|F\|_{\tilde{Z}_t}.
\end{equation}
Actually, one can prove that  for $N\geq 7$  
\begin{equation} \label{disp}
\|e^{-it\partial^{2}_{x}}F\|_{L^{\infty}_{x}\HH^{1}_{y}}\leq C 
\langle t \rangle^{-1/2}
\Big(
\|F\|_{Z}^2+\langle t\rangle^{-\frac{1}{4}}\big(\|x F\|_{L_{x,y}^2}^2+\|F\|_{H_{x,y}^N}^2\big)
\Big)^{1/2},
\end{equation}
which implies~\eqref{est7}. We refer to~\cite[Lemma 7.3]{HPTV} for a proof of~\eqref{disp} in a similar context. \medskip

 $\bullet$ The proof of~\eqref{est41} follows from~\eqref{N0L2}, the second part of Lemma~\ref{ControlSS+}, and the fact that 
$$
\|xF\|_{Z}\lesssim T^{-\delta}\|F\|_{S^+}+T^{2\delta}\|F\|_{S}
$$ 
whose elementary proof is given in~\cite[Estimate (3.31)]{HPTV}.\medskip

$\bullet$ We now turn to the proof of~\eqref{di1} and~\eqref{differNRS}.  We first decompose
$$
F=F_c +F_f, \quad G=G_c+G_f, \quad H=H_c+H_f
$$
where $F_c=F\varphi(t^{-1/4}x)$ and $G_c, H_c$ are similarly defined. We claim that the main contribution to~\eqref{di1} and~\eqref{differNRS} comes from the ``close" components $F_c, G_c,$ and $H_c$. Indeed, we show that 
\begin{multline}\label{aux32}
\big\|\mathcal N_0^{t}[F,G,H]-\mathcal N_0^{t}[F_c, G_c, H_c]\big\|_{Z}+t^{-1}\big\|\mathcal R[F,G,H]-\mathcal R[F_c, G_c, H_c]\big\|_{Z}\lesssim\\[3pt]
 \lesssim (1+|t|)^{-\frac{33}{32}}\|F\|_{S}\|G\|_{S}\|H\|_{S}
\end{multline}
and 
\begin{multline}\label{aux33}
\big\|\mathcal N_0^{t}[F,G,H]-\mathcal N_0^{t}[F_c, G_c, H_c]\big\|_{S}+t^{-1}\big\|\mathcal R[F,G,H]-\mathcal R[F_c, G_c, H_c]\big\|_{S}\lesssim \\[3pt]
 \lesssim (1+|t|)^{-\frac{33}{32}}\|F\|_{S^+}\|G\|_{S^+}\|H\|_{S^+}.
\end{multline}
Estimate~\eqref{aux32} follows from~\eqref{ZSNorm} and the following inequalities: Here $\widetilde G$ denotes either $G$ or $G_f$ and similarly for $\widetilde H$.
\begin{align*}
\|\mathcal N_0^{t}[F_f,\widetilde G, \widetilde H]\|_{S}+t^{-1}\|\mathcal R[F_f,\widetilde G, \widetilde H]\|_{S}\lesssim& (1+|t|)^{-1}\|F\|_{S}\|G\|_{S}\|H\|_{S}\\[3pt]
\|\mathcal N_0^{t}[F_f,\widetilde G, \widetilde H]\|_{L^2_{x,y}}+t^{-1}\|\mathcal R[F_f,\widetilde G, \widetilde H]\|_{L^2_{x,y}}\lesssim& (1+|t|)^{-1}\|F_f\|_{L^2}\|G\|_{S}\|H\|_{S}\\
\lesssim&(1+|t|)^{-5/4}\|F\|_{S}\|G\|_{S}\|H\|_{S}.
\end{align*}
These inequalities follow directly from~\eqref{N0L2},~\eqref{a0} and~\eqref{a2}. Similarly, estimate~\eqref{aux33} follows from~\eqref{est4},~\eqref{a2}, and the fact that $\|F_f\|_{S}\lesssim t^{-1/4}\|F\|_{S^+}$. The contribution of the terms involving $G_f$ and $H_f$ is treated similarly.

As a consequence,~\eqref{di1} and~\eqref{differNRS} reduce to:
\begin{align}
\big\| \mathcal N_0^{t}[F_c,G_c, H_c]-\frac{\pi}{t}\mathcal R[F_c, G_c, H_c]\big\|_{Z}\lesssim& (1+|t|)^{-1-10\delta}\|F\|_{S}\|G\|_{S}\|H\|_{S}\label{d1}\\[3pt]
\big\| \mathcal N_0^{t}[F_c,G_c, H_c]-\frac{\pi}{t}\mathcal R[F_c, G_c, H]_c\big\|_{S}\lesssim& (1+|t|)^{-1-20\delta}\|F\|_{S^+}\|G\|_{S^+}\|H\|_{S^+}.\label{d2}
\end{align}

For this, we will borrow the following lemma from~\cite[Lemma 3.10]{HPTV}. 
\begin{lemma}\label{lemma3.10}
Let $t\geq 1$ and suppose that $f^{a}, f^{b}, f^{c} \in L^2(\R)$ are supported in the set $\big\{|x|\leq t^{1/4}\big\}$. Then there holds that for any integer $m$
\begin{multline*} 
|\xi|^m \left|\int_{\R^2} e^{2it\eta \kappa} \widehat f^{a}(\xi-\eta) \overline{\widehat f^{b}}(\xi-\eta-\kappa) \widehat f^{c}(\xi-\kappa) d\eta d\kappa -\frac{\pi}{t}\widehat f^{a}(\xi)\widehat{f^{b}}(\xi) \widehat{f^{c}}(\xi)\right|\lesssim \\
\lesssim t^{-\frac{11}{10}}\min_{\sigma \in \mathfrak G_3} \|f^{\sigma(a)}\|_{H^m}\|f^{\sigma(b)}\|_{L^2}\|f^{\sigma(c)}\|_{L^2}.
\end{multline*}
\end{lemma}

With this lemma in hand, we can proceed with the proof of~\eqref{d1} and~\eqref{d2}. Again we obtain the $Z-$norm estimate~\eqref{d1} using~\eqref{ZSNorm},~\eqref{est4},~\eqref{a2}, and the fact that 
\begin{align*}
&\big\| \mathcal N_0^{t}[F_c,G_c, H_c]-\frac{\pi}{t}\mathcal R[F_c, G_c, H_c]\big\|_{L^2_{x,y}}\\
&=\left\|(2\pi)^{-1}\int_{-\pi}^{\pi}e^{-is\HH} \left[\int_{\R^2}e^{2it\eta \kappa} \widehat{e^{is\HH}F_c}(\xi-\eta) \overline{\widehat{e^{is\HH}G_c}}(\xi-\eta-\kappa) \widehat{e^{is\HH}H_c}(\xi-\kappa) \, d\eta\, d\kappa\right. \right.\\
 &\left. \left.\qquad\qquad \qquad \qquad \qquad\qquad \qquad -\frac{\pi}{t}\widehat{e^{is\HH}F_c}(\xi) \overline{\widehat{e^{is\HH}G_c}}(\xi) \widehat{e^{is\HH}H_c}(\xi)\right]\,ds\right\|_{L_{\xi, y}^2}\lesssim\\
 &\lesssim t^{-\frac{11}{10}}\left\|\int_{-\pi}^{\pi}\|e^{is\HH}F_c(\cdot, y)\|_{H_x^2} \|e^{is\HH}G_c(\cdot, y)\|_{L_x^2} \|e^{is\HH}H_c(\cdot, y)\|_{L_x^2}ds\right\|_{L^2_y}\lesssim t^{-\frac{11}{10}}\|F\|_{S}\|G\|_{S}\|H\|_{S}.
\end{align*}
Arguing as above and using~\eqref{HandT} one also obtains~\eqref{d2}. \medskip

$\bullet$ The estimates~\eqref{esti1} and~\eqref{est2} are obtained in a similar way to~\eqref{est4} and~\eqref{est41}. we do not write the details. \medskip

This ends the proof of Lemma~\ref{lem46}.
\end{proof}

\subsection{Conclusion of the proof of Proposition~\ref{StrucNon}} We decompose the nonlinearity as follows
\begin{multline*}
 \mathcal{N}^{t}[F,G,H]=\sum_{\substack{A,B,C\\\max(A,B,C)\geq T^{\frac{1}{6}}}}\mathcal{N}^{t}[Q_{A}F,Q_{B}G,Q_CH]+\widetilde{\mathcal N^{t}}[Q_{\leq T^{\frac{1}{6}}}F,Q_{\leq T^{\frac{1}{6}}}G,Q_{\leq T^{\frac{1}{6}}}H]+\\
 +
{\mathcal N}_{0}^{t}[Q_{\leq T^{\frac{1}{6}}}F,Q_{\leq T^{\frac{1}{6}}}G,Q_{\leq T^{\frac{1}{6}}}H]\,.
 \end{multline*}
The first term above contributes to $\mathcal E_{1}$ by Lemma~\ref{BilEf}. The second term is the one studied in  Lemma~\ref{FastOsLem}, and can therefore be written as $\widetilde{\mathcal E_{1}}+\mathcal E_{2}$ with $\widetilde{\mathcal E_{1}}$ giving an acceptable contribution to $\mathcal E_{1}$. Now, decompose the third term  as
\begin{multline*}
{\mathcal N}_{0}^{t}[Q_{\leq T^{\frac{1}{6}}}F,Q_{\leq T^{\frac{1}{6}}}G,Q_{\leq T^{\frac{1}{6}}}H]=\frac{\pi}{t}\mathcal R[F,G,H]-\frac{\pi}{t}\sum_{\substack{A,B,C\\\max(A,B,C)\geq T^{\frac{1}{6}}}}\mathcal{R}[Q_{A}F,Q_{B}G,Q_CH]\\
\qquad+\Big({\mathcal N}_{0}^{t}[Q_{\leq T^{\frac{1}{6}}}F,Q_{\leq T^{\frac{1}{6}}}G,Q_{\leq T^{\frac{1}{6}}}H]-\frac{\pi}{t}\mathcal{R}[Q_{\leq T^{\frac{1}{6}}}F,Q_{\leq T^{\frac{1}{6}}}G,Q_{\leq T^{\frac{1}{6}}}H]\Big).
\end{multline*} 
By using the first four estimates in Lemma~\ref{lem46}, we check that  the   third term  contributes to $\mathcal{E}_{1}$.  The second term is estimated as in Lemma~\ref{BilEf}.
This finishes the proof of Proposition~\ref{StrucNon}.

\section{Proof of the main results} \label{sect5}

We give here the main lines of the proofs of Theorems~\ref{ModScatThm} and~\ref{ExMWO}. We do not give all the details, since the argument is close to the corresponding results in~\cite{HPTV}.

\subsection{Modified wave operators}
 
 The next result implies Theorem~\ref{ModScatThm}
\begin{theorem}\label{thm51}
There exists $\varepsilon>0$ such that if $U_0\in S^+$ satisfies
\begin{equation*} 
\Vert U_0\Vert_{S^+}\le\varepsilon,
\end{equation*}
and if $\widetilde{G}$ is the solution of~\eqref{RSS} with initial data $U_0$, then there exists $U$ a solution of~\eqref{CNLS} such that
$e^{-it\mathcal{D}}U(t)\in \mathcal{C}\big([0,+\infty); S\big)$ and
\begin{equation*}
\begin{split}
\Vert e^{-it\mathcal{D}}U(t)-\widetilde{G}(\pi\ln t)\Vert_{S}\longrightarrow 0\,\;\;\hbox{ as }\,\;\;t\longrightarrow+\infty.
\end{split}
\end{equation*}
\end{theorem}

\begin{proof} The proof is similar to the proof of~\cite[Theorem 5.1]{HPTV}, and we refer to it for the details.  Set $G(t)=\widetilde{G}(\pi\ln t)$, $F(t)= e^{-it\mathcal{D}}U(t)-G(t)$ and  define the mapping
\begin{equation*}
\Phi(F)(t)=i\int_t^\infty\left\{\mathcal{N}^\sigma[F+G,F+G,F+G]-\frac{\pi}{\sigma}\mathcal{R}[G,G,G]\right\}d\sigma.
\end{equation*}
To prove Theorem~\ref{thm51} it is enough to find a fixed point for $\Phi$ in a suitable space. Actually,  using Proposition~\ref{StrucNon}, we can show that  for  $\varepsilon$ small enough,   there exists $\varepsilon_{1}$ such that $\Phi$ defines a contraction on the complete metric space $\mathfrak{A}$ defined by
\begin{equation*}
\begin{split}
\mathfrak{A}:=&\big\{F\in \mathcal{C}^1\big([1,+\infty);S\big)\,:\,\,\Vert F\Vert_\mathfrak{A}\le\varepsilon_{1}\big\}\\[3pt]
\Vert F\Vert_\mathfrak{A}:=&\sup_{t\geq 1}\left\{(1+\vert t\vert)^\delta\Vert F(t)\Vert_{S}+(1+\vert t\vert)^{2\delta}\Vert F(t)\Vert_{Z}+(1+\vert t\vert)^{1-\delta}\Vert \partial_tF(t)\Vert_{S}\right\}.
\end{split}
\end{equation*}
This defines $e^{-it\mathcal{D}}U(t)$ for $t\geq 1$, and this function can be continued for $t\in [0,1]$ in such a way that $e^{-it\mathcal{D}}U(t)\in \mathcal{C}\big([0,+\infty); S\big)$.
\end{proof}
\subsection{Small data scattering}\label{SMS}

We now state a more precise version of Theorem~\ref{ModScatThm} which is the main result of this paper.

\begin{theorem}\label{ModScatThm2}
Let $U_0\in S^+$ such that  $\Vert U_0\Vert_{S^+}\le\varepsilon$ and consider the corresponding solution $U$ of~\eqref{CNLS}. Then if  $\varepsilon>0$ is small enough,  there exists a solution $G$ of~\eqref{RSS}  so that 
\begin{equation*} 
\Vert e^{-it\mathcal{D}}U(t)-G(\pi\ln t)\Vert_{S}\longrightarrow 0\qquad\hbox{ as \;\;}t\longrightarrow+\infty.
\end{equation*}
\end{theorem}

\begin{proof} We can follow the proof of~\cite[Theorem 6.1]{HPTV}, and we only give here the main steps. Define  $F(t)=e^{-it\mathcal{D}}U(t)$.\medskip

{\bf Step 1:}
 Let $U_0\in S^+$ such that  $\Vert U_0\Vert_{S^+}\le\varepsilon$. Then  if  $\varepsilon>0$  there is a global solution of~\eqref{CNLS} which satisfies for all $T>0$
 \begin{equation}\label{Beps} 
\Vert F(t)\Vert_{X^+_T}\leq 2\varepsilon.
\end{equation}
 There is no particular difficulty in obtaining this global existence result with small initial condition. The proof essentially relies on a bootstrap (continuous induction) argument, using the estimates of Proposition\;\ref{StrucNon}. However, it's worth mentioning that it is at this point that our argument is not simply perturbative as it essentially relies on the conservation of the $Z-$norm of the limit system~\eqref{RS}.
 
  Observe however, that in order to get the bound of the $Z-$norm one uses the crucial cancelation $\<i\mathcal{F}_{x}\big(\mathcal{R}(F,F,F)\big)(\xi),\mathcal{F}_{x}(F)(\xi)\>_{\HH^{1}_{y}\times \HH^{1}_{y}}=0$, for all $\xi \in \R$. \medskip
 
  {\bf Step 2:}
Define $T_n=e^{n/\pi}$ and $G_n(t)=\widetilde{G}_n(\pi\ln t)$, where $\widetilde{G}_n$ solves~\eqref{RSS} with Cauchy data such that $\widetilde{G}_n(n)=G_n(T_n)=F(T_n)$. With the information we have on the limit system, we can prove that  for all $t\geq T_n$,
\begin{equation*} 
\Vert G_n(t)\Vert_{Z}+(1+\vert t\vert)^{-\delta}\Vert G_n(t)\Vert_{S}+(1+\vert t\vert)^{-5\delta}\Vert G_n(t)\Vert_{S^+}+(1+\vert t\vert)^{1-\delta}\Vert \partial_tG_n(t)\Vert_{S}\lesssim\varepsilon
\end{equation*}
uniformly in $n\geq 0$. \medskip

 {\bf Step 3:}
Set $\kappa=\delta/\pi$. With Gronwall, we can prove that
\begin{equation}\label{Bds}
\sup_{T_n\leq t\leq T_{n+4}}\Vert F(t)-G_n(t)\Vert_{S}\lesssim \varepsilon^3 e^{-n\kappa}.
\end{equation}
  
 {\bf Step 4:} In particular, the above implies that $\|G_{n+1}(T_{n+1})-G_n(T_n)\|_S\lesssim \varepsilon^3 e^{-n\kappa}$. Thanks to~\eqref{stab} it easy to deduce  that 
\begin{equation*}
\Vert \widetilde{G}_n(0)-\widetilde{G}_{n+1}(0)\Vert_{S}\lesssim \varepsilon^3e^{-n\kappa/2}.
\end{equation*}
Observe here that one applies~\eqref{stab} with $t$ replaced by $\ln t$. This is crucial, and that's why we still have an exponentially decaying bound in the previous line.\medskip

Therefore, we see that $\{\widetilde{G}_n(0)\}_n$ is a Cauchy sequence in $S$ and  converges to an element $G_{0,\infty}\in S$ which satisfies 
\begin{equation*}
\Vert G_{0,\infty}\Vert_{Z}\lesssim\varepsilon,\qquad \Vert \widetilde{G}_n(0)-G_{0,\infty}\Vert_{S}\lesssim \varepsilon^3e^{-n\kappa/2}.
\end{equation*}
Denote by $G_\infty(t)=\widetilde{G}_\infty(\pi\ln t)$ with $\widetilde{G}_\infty$ the solution of~\eqref{RSS} with initial data $\widetilde{G}_\infty(0)=G_{0,\infty}$, then by~\eqref{stab} we obtain 
\begin{equation*}
\sup_{[0,T_{n+2}]}\Vert G_\infty(t)-G_n(t)\Vert_{S}\lesssim \varepsilon^3 e^{-n\kappa/4}.
\end{equation*}
From this bound and~\eqref{Bds} we deduce that 
\begin{equation*}
\sup_{T_n\leq t\leq T_{n+1}}\Vert G_\infty(t)-F(t)\Vert_{S} \lesssim \varepsilon^3e^{-n\kappa/4},
\end{equation*}
which shows that $G_{\infty}$ satisfies the result.
\end{proof}

\subsection{Particular dynamics of NLS}
 We first state a result, which has its own interest and  which links the dynamics of~\eqref{CNLS} to the dynamics of the resonant Schr\"odinger equation~\eqref{model}. Actually, we consider here the approximation of NLS with the nonlinearity $\mathcal{N}^{t}_{0}$ instead of $\frac{\pi}{t}\mathcal{R}$. Recall that $\mathcal{N}^{t}_{0}$ is defined in~\eqref{defN0},~\eqref{defPi} and that we have 
\begin{equation*}
\begin{split}
\mathcal{F}_x\mathcal{N}_{0}^{t}[F,G,H](\xi,y)
=& \sum_{(p,q,r,s)\in \Gamma_{0}} \Pi_{p}\int_{\mathbb{R}^2}e^{2it\eta\kappa} \widehat{  F}_{q}(\xi-\eta,y)\overline{\widehat{ G}_{r}}(\xi-\eta-\kappa,y)\widehat{  H}_{s}(\xi-\kappa,y)d\eta d\kappa \\
=&\int_{\mathbb{R}^2}e^{2it\eta\kappa} \mathcal T \left(\widehat{  F}(\xi-\eta,y),{\widehat{ G}}(\xi-\eta-\kappa,y),\widehat{  H}(\xi-\kappa,y)\right)d\eta d\kappa.
\end{split}
\end{equation*}
 Consider the equation
  \begin{equation}\label{eq53}
 i\partial_{t}W=\mathcal{N}^{t}_{0}[W,W,W], \quad (t,x)\in \R\times \R^d.
 \end{equation}
 \begin{theorem}\label{thm52}
There exists $\varepsilon>0$ such that if $U_0\in S^+$ satisfies
\begin{equation*} 
\Vert U_0\Vert_{S^+}\le\varepsilon,
\end{equation*}
and if $W$ is the solution of~\eqref{eq53} with initial data $U_0$, then there exists $U$ a solution of~\eqref{CNLS} such that
$e^{-it\mathcal{D}}U(t)\in \mathcal{C}\big([0,+\infty); S\big)$ and
\begin{equation*}
\begin{split}
\Vert e^{-it\mathcal{D}}U(t)-{W}(t)\Vert_{S}\longrightarrow0\,\;\;\hbox{ as }\,\;\;t\longrightarrow+\infty.
\end{split}
\end{equation*}
\end{theorem}

\begin{proof} Define  $F(t)= e^{-it\mathcal{D}}U(t)-W(t)$. Is it enough to show that  the mapping
\begin{equation*}
\Phi(F)(t)=i\int_t^\infty\big\{\,\mathcal{N}^\sigma[F+W,F+W,F+W]-\mathcal{N}^{t}_{0}[W,W,W]\,\big\}d\sigma,
\end{equation*}
defines a   contraction in the space $\mathfrak{A}$. We can proceed as in Theorem~\ref{thm51} and with the estimates of Lemma~\ref{lem46}. We also refer to~\cite[Theorem 5.1]{HPTV} for the details.
\end{proof}

\subsubsection{ Proof of Proposition~\ref{prop15}}  Assume that $\kappa_{0}=1$. For some $\psi$ define
\begin{eqnarray*}
W(t,x,y)&:=&e^{-it\DD}\big(\psi(t,x)e^{2i(n+1)t}(y_{1}+iy_{2})^{n}e^{-\frac{|y|^{2}}2}\big)=\big(e^{it\partial^{2}_{x}} \psi(t,x)\big)(y_{1}+iy_{2})^{n}e^{-\frac{|y|^{2}}2}\\
&:=&f(t,x)g_{n}(y).
\end{eqnarray*}

$\bullet$ For the first point of the proposition, by Theorem~\ref{thm52}, it is enough to check that  $W$ defined above is solution to~\eqref{eq53}. Recall the definition~\eqref{defI} of $\mathcal{I}^{t}$, then using the definition~\eqref{RS} of  $\mathcal{T}$
\begin{equation*}
\mathcal{N}^{t}_{0}[W,W,W]=\mathcal{I}^{t}[f,f,f]\mathcal{T}[g_{n},g_{n},g_{n}]=\lambda_{n}g_{n}\mathcal{I}^{t}[f,f,f],
\end{equation*}
with $\lambda_{n}=\Vert  g_{n}\Vert^4_{L^4(\R^2)}/\Vert  g_{n}\Vert^2_{L^2(\R^2)}$ by~\eqref{gn}. Now, the equation~\eqref{eq53} is equivalent to $i\partial_{t}f=\lambda_{n}\mathcal{I}^{t}[f,f,f]$, which also reads $(i\partial_{t}-\partial^{2}_{x})\psi=\lambda_{n}|\psi|^{2}\psi$, but this is the case by assumption.
\medskip

$\bullet$ For the second point, we use that $i\partial_{t} \wh{G}=\frac{\pi}t \mathcal{T}[\wh{G},\wh{G},\wh{G}]$ is satisfied iff $i\partial_{t} \wh{f}=\frac{\pi \lambda_{n}}t |\wh{f}|^{2}\wh{f}$, which corresponds to $(i\partial_t -\partial^2_{x}) \psi=  \frac{\pi \lambda_n}{t}   \mathcal{F}_{\xi \to x}^{-1}\big(\vert \wh\psi\vert^2\wh\psi\big)$.
\medskip

$\bullet$ Concerning the symmetry assumption : we use that $g_{n}(R_{\theta}y)=e^{in\theta}g_{n}(y)$ and that $\Delta (U\circ R_{\theta})=(\Delta U)\circ R_{\theta}$.
\subsubsection{ Proof of Proposition~\ref{prop16}} We set $g=g_{+}$. For $\psi_{\pm}$ which satisfy the system~\eqref{syst1}, we define $f_{\pm}=e^{it\partial^{2}_{x}}\psi_{\pm}$. Then by Theorem~\ref{thm52}, it is enough to prove that $f_{+}g+f_{-}\ov{g}$ solves~\eqref{eq53}. 

First, we  claim that  $\TT[g,\ov{g},g]=0$. Actually, by definition \eqref{RS} of $\TT$ there exists $a,b\in \C$ such that  $\TT[g,\ov{g},g]=ag+b\ov{g}$, because $E_{1}=\operatorname{Span}\{g,\ov{g}\}$. Then, taking the scalar product with~$g$  we get $a=\int{\ov{g}g^{3}}=0$, and similarly $b=0$. The same argument also gives  $\TT[g,{g},g]= g/4$ (for more properties of $\TT$ we refer to \cite{GHT1}). 

Then by symmetry considerations, we get that~\eqref{eq53} is equivalent to  
$$ i\partial_{t} f_{+}=\lambda_{1}\big(\mathcal{I}^{t}(f_{+},f_{+},f_{+})+2\mathcal{I}^{t}(f_{+},f_{-},f_{-})\big),$$
 and an analogous equation for $f_{-}$, and this is the case if we use that $f_{\pm}=e^{it\partial^{2}_{x}}\psi_{\pm}$.
 
\subsubsection{ Proof of Proposition~\ref{prop17} }

Let $1\leq r(\xi)\leq 2$ be a smooth function and $\phi\in \mathcal{S}(\R)$. If $g(t,y)$ solves~\eqref{star}, then for all $\mu>0$, $\mu g(\mu^{2}t,\mu y)$ also solves~\eqref{star}. As a consequence,  $g(t,y,\xi)=r(\xi)e^{-|y|^2r^2(\xi)/2}e^{ir^2(\xi)t/2}$ is solution to~\eqref{star} and thus we get that
\begin{equation}\label{111}
\wh{G}(t,\xi,y)=\eps \phi(\xi) r(\xi)e^{-y^2r^2(\xi)/2}e^{i\eps^2\phi^2(\xi)r^2(\xi)t/2},
\end{equation}
is solution to~\eqref{RSS}. Therefore if $\eps>0$ is small enough, by Theorem~\ref{ExMWO} there exists  $U\in \mathcal{C}\big([0,+\infty); H^N(\R\times \R^2)\big)$ to~\eqref{CNLS} so that  
\begin{equation*}
\Vert  e^{-it\DD}U(t)-     G(\pi\ln t)  \Vert_{L^2(\R\times \R^2)}  \longrightarrow 0.
\end{equation*}
We proceed by contradiction. By Parseval in $x$,
\begin{multline*}
\Vert \wh{G}(\pi \ln t,\xi,y)-e^{-it(\xi^{2}+\HH_{2})} \sum_{j=1}^{n}\wh{\psi}_{j}(t,\xi)f_{j}(t,y)   \Vert_{L^2(\R\times \R^2)} =\\
\begin{aligned}
&=\Vert {G}(\pi \ln t,x,y)-e^{-it \mathcal{D}} \sum_{j=1}^{n} {\psi}_{j}(t,x)f_{j}(t,y)   \Vert_{L^2(\R\times \R^2)} \\
&\leq \Vert {G}(\pi \ln t)-e^{-it \mathcal{D}}U(t)  \Vert_{L^2(\R\times \R^2)} +\Vert  U(t)-  \sum_{j=1}^{n} {\psi}_{j}(t)f_{j}(t)   \Vert_{L^2(\R\times \R^2)},
\end{aligned}
\end{multline*}
and we assume that there is a choice of $\psi_{j}$ and $f_{j}$ so that the previous line tends to 0 when $t\longrightarrow+\infty$. But then by~\eqref{111} we have
\begin{equation*}
\big\Vert  \eps \phi(\xi)r(\xi)e^{-y^2r^2(\xi)/2}  -   \sum_{j=1}^{n} \Phi_{j}(t,\xi)F_{j}(t,y)    \big\Vert_{L^2(\R\times \R^2)} \longrightarrow 0,
\end{equation*}
with $\Phi_{j}(t,\xi):=e^{-i\eps^2\phi^2(\xi)r^2(\xi)\pi \ln t/2}e^{-it\xi^{2}}\wh{\psi}_{j}(t,\xi)$, $F_{j}(t,y):=e^{-it\HH_{2}}f_{j}(t,y)$ and by assumption $\|\Phi_{j}(t,\xi)\|_{L^{2}(\R)},\|F_{j}(t,y)\|_{L^{2}(\R^{2})}\lesssim 1$, uniformly in $t\in \R$. Therefore, there exists a sequence $t_{k}\longrightarrow +\infty$ such that  for all $1\leq j\leq n$, $\Phi_{j}(t_{k},\xi) \rightharpoonup \Phi^{\infty}_{j}(\xi)$ and $F_{j}(t_{k},y) \rightharpoonup F^{\infty}_{j}(y)$. This is a contradiction, since $\eps \phi(\xi)r(\xi)e^{-y^2r^2(\xi)/2} $    is clearly not a sum of separate variable functions.

\section{Transfer  Principle} \label{sect6}


  The goal of this section is to present and prove a lemma that will allow us to transfer $L^2$ estimates on operators into estimate in $S$ and $S^+$ norms. All the trilinear operators that appear in this paper are of the form:
\begin{multline}\label{triop}
\mathfrak T_m[F, G, H](x,y):=\\
=\mathcal F_x^{-1}\iint_{\R^{2}}m(\xi, \eta, \kappa) \mathcal{K}[\widehat F(\xi-\eta, \cdot), \widehat G(\xi-\eta-\kappa,\cdot), \widehat H(\xi-\kappa, \cdot)](y) \, d\eta d\kappa
\end{multline}
where $\mathcal K$ is a trilinear operator on functions from $\R^{d}_y\to \C$. Actually, we are in one of the following cases
\begin{enumerate}[\qquad (1)]
\item Define the operator $\mathcal J$   for three functions $\mathfrak f,\mathfrak g, \mathfrak h: \R_y^{d}\to \C$ as
\begin{equation*} 
\mathcal J^{t}[\mathfrak f,\mathfrak g, \mathfrak h]= e^{-it\HH_{d}}\left(e^{-it\HH_{d}} \mathfrak f \,e^{-it\HH_{d}}\mathfrak g\, e^{-it\HH_{d}} \mathfrak h\right).
\end{equation*}
Then in Lemmas~\ref{wup} and~\ref{BilEf} we estimate
\begin{equation*}
\mathcal N^{t}[F, G, H](x,y)= \mathcal F_x^{-1}\iint_{\R^{2}} e^{2it \eta\kappa}\mathcal{J}^{t}[\widehat F(\xi-\eta, \cdot), \widehat G(\xi-\eta-\kappa,\cdot), \widehat H(\xi-\kappa, \cdot)](y) \, d\eta d\kappa.
\end{equation*}

\item Recall that the operator $\mathcal R$ is defined in~\eqref{RSS}. Then in Lemma~\ref{lem46} we consider
\begin{equation*}
\mathcal N_{0}^{t}[F, G, H](x,y)= \mathcal F_x^{-1}\iint_{\R^{2}} e^{2it \eta\kappa}\mathcal{R}[\widehat F(\xi-\eta, \cdot), \widehat G(\xi-\eta-\kappa,\cdot), \widehat H(\xi-\kappa, \cdot)](y) \, d\eta d\kappa.
\end{equation*}
\item We have to estimate the operators $\mathcal{O}^{t}_{1}$ and $\mathcal{O}^{t,\tau}_{2}$ defined in~\eqref{defO1} and~\eqref{defO2} respectively.\vspace{5pt}
\item  In Lemma~\ref{lem46} we have to estimate a difference of  previous quantities.
\end{enumerate}

We define a LP-family $\widetilde{Q}=\{\widetilde{Q}_A\}_A$ to be a family of operators (indexed by the dyadic integers) of the form
\begin{equation*}
\widehat{\widetilde{Q}_{1}f}(\xi)=\widetilde{\varphi}(\xi)\widehat{f}(\xi),\qquad \widehat{\widetilde{Q}_Af}(\xi)=\widetilde{\phi}(\frac{\xi}{A})\widehat{f}(\xi), \quad A\geq 2
\end{equation*}
for two smooth functions $\widetilde{\varphi},\widetilde{\phi}\in \mathcal{C}^\infty_c(\mathbb{R})$ with $\widetilde{\phi}\equiv0$ in a neighborhood of $0$. We also define the set of admissible transformations to be the family of operators $\{T_B\}$ where for any $B$,
\begin{equation*}
T_B=\lambda_B\widetilde{Q}_B,\qquad \vert\lambda_B\vert\leq 1
\end{equation*}
for some LP-family $\widetilde{Q}$.

Given a multiplier $m(\xi, \eta, \kappa)$ and a subset $\Lambda$ of 4-tuples of dyadic integers, we define a localization of $\mathfrak T_m$ as the operator:
$$
\mathfrak T_m^\Lambda[F, G, H]=\sum_{(A,B, C,D)\in \Lambda} T_A\mathfrak T_m[T'_B F, T''_CG, T'''_D H]
$$
where $T, T', T'',$ and $T'''$ are admissible operators. It's not hard to see that $\mathfrak T_m^\Lambda$ is also formally of the type~\eqref{triop}. Note that a localization is determined by the set $\Lambda$ and the admissible operators $T, T',T'', T'''$.

 Finally, we say that a norm $\mathcal{B}$ is admissible if
\begin{enumerate}[\qquad (1)]
\item For any admissible transformation $T=\{T_A\}_A$, there holds that
\begin{equation*} 
\Vert \sum_A T_AF\Vert_{\mathcal{B}}\lesssim \Vert F\Vert_{\mathcal{B}}.
\end{equation*}
\item If $\Xi=(\partial_{y_{1}}, \ldots, \partial_{y_d}, y_{1},\ldots, y_d)$, $\alpha\in \N^{2d}$ is a multi-index, and $\psi \in \mathcal{C}_c^\infty(\R)$ then 
\begin{equation*}
\|\Xi^\alpha \psi(\nu^{-2}\HH) f\|_{\mathcal B}\lesssim_{\psi} \nu^{|\alpha|}\|f\|_{\mathcal B}.
\end{equation*}
\end{enumerate}

All norms that we consider are admissible. In particular, the second point is a consequence of the equivalence of norms~\eqref{eq} and Lemma~\ref{lem.com}.\medskip

As remarked in~\cite{HPTV}, we have the following Leibniz rule for $\mathcal{I}^{t}[f,g,h]$, namely
\begin{equation} \label{leib}
Z\mathcal{I}^{t}[f,g,h]=\mathcal{I}^{t}[Zf,g,h]+\mathcal{I}^{t}[f,Zg,h]+\mathcal{I}^{t}[f,g,Zh],\quad Z\in\{ix,\partial_x\}.
\end{equation}
For $Z=\partial_x$ this follows from the fact that  for all $t\in \R$, $[\UU(t),\partial_x]=0$. When $Z=ix$, we use the relation $x\UU(t)=\UU(t)(x+2it\partial_x)$. This implies that if  $m=m(\eta, \kappa)$ in~\eqref{triop} is independent of $\xi$, then  $Z\mathfrak{T}[F,G,H]= \mathfrak{T}[ZF,G,H]+\mathfrak{T}[F,ZG,H]+\mathfrak{T}[F,G,ZH]$, 
for $Z\in\{ix,\partial_x\}$.

 One also has the relation Leibniz-type formula:
  \begin{equation}\label{HandT}
\HH \mathcal R[F, G, H]=\mathcal R[\HH F, G, H]-\mathcal R[F, \HH G, H]+\mathcal R[F, G,\HH H].
\end{equation}
Thanks to these relations, we are able to dispatch the derivatives in the tri-linear term.

\begin{lemma}\label{ControlSS+}
Let $\mathfrak T_m$ be an operator as in~\eqref{triop} with $m=m(\eta, \kappa)$ being independent of $\xi$. Let $\Lambda$ be a set of $4$-tuples of dyadic integers and suppose that for all localizations of $\mathfrak{T}_m$ to $\Lambda$, we have
\begin{equation*} 
\Vert \mathfrak{T}_m^\Lambda[F^{a},F^{b},F^{c}]\Vert_{L^2_{x,y}}\leq K\min_{\sigma\in\mathfrak{S}_3}\Vert F^{\sigma(a)}\Vert_{L_{x,y}^2}\Vert F^{\sigma(b)}\Vert_{\mathcal{B}}\Vert F^{\sigma(c)}\Vert_{\mathcal{B}}
\end{equation*}
for some admissible norm $\mathcal{B}$. Then, for any localization $\mathfrak{T}_m$ at $\Lambda$,
\begin{equation}\label{CS}
\begin{split}
\Vert \mathfrak{T}_m^\Lambda[F^{a},F^{b},F^{c}]\Vert_{S}&\lesssim K\max_{\sigma\in\mathfrak{S}_3}\Vert F^{\sigma(a)}\Vert_{S}\Vert F^{\sigma(b)}\Vert_{\mathcal{B}}\Vert F^{\sigma(c)}\Vert_{\mathcal{B}}.
\end{split}
\end{equation}
Assume in addition that, for $Y\in\{x,(1-\partial_{xx})^4\}$,
\begin{equation}\label{FL2}
\Vert YF\Vert_{\mathcal{B}}\lesssim \theta_{1}\Vert F\Vert_{S^+}+\theta_{2}\Vert F\Vert_{S},
\end{equation}
then for all realizations of $\mathfrak{T}_m$ at $\Lambda$,
\begin{equation}\label{CS+}
\begin{split}
\Vert \mathfrak{T}_m^\Lambda[F^{a},F^{b},F^{c}]\Vert_{S^+}&\lesssim K\max_{\sigma\in\mathfrak{S}_3}\Vert F^{\sigma(a)}\Vert_{S^+}\big(\Vert F^{\sigma(b)}\Vert_{\mathcal{B}}+\theta_{1}\Vert F^{\sigma(b)}\Vert_{S}\big) \Vert F^{\sigma(c)}\Vert_{\mathcal{B}}\\
&\quad+\theta_{2}K\max_{\sigma\in\mathfrak{S}_3}\Vert F^{\sigma(a)}\Vert_{S}\Vert F^{\sigma(b)}\Vert_{S}\Vert F^{\sigma(c)}\Vert_{\mathcal{B}}.
\end{split}
\end{equation}
\end{lemma}

\begin{proof}
We start with~\eqref{CS}: \medskip

 {\bf Bound of the $\boldsymbol{L^{2}}$ component of the $\boldsymbol{S}$ norm:} At first, we notice that if $Q_{A}$ is a LP family, then
$$
[x, Q_{A}]=A^{-1}Q'_A
$$
where $Q'_A$ is another LP family. Obviously, $T_A=A^{-1}Q'_A$ is an admissible transformation. As a result, from~\eqref{leib} we have that:
\begin{equation}\label{xtm}
\begin{split}
x\mathfrak T_m^\Lambda[F,G,H]=&\mathfrak T_m^\Lambda[xF,G,H]+\mathfrak T_m^\Lambda[F,xG,H]+\mathfrak T_m^\Lambda[F,G,xH]+\mathfrak T^\Lambda_m[F, G, H]
\end{split}
\end{equation}
where $\mathfrak T_m^\Lambda$ appearing on the right-hand side above is a different localization of $\mathfrak T_m^\Lambda$. From this the bound on the weighted component of the $S$ norm follows. \medskip

 {\bf Bound of the $\boldsymbol{H^{N}}$ component of the $\boldsymbol{S}$ norm:} Let $\mathcal P_N$ denote the Littlewood-Paley projection either the $x-$direction or $y-$direction. We can decompose
$$
\mathcal P_M\mathfrak T_m^\Lambda=\mathcal P_M \mathfrak T_{m, M, low}^\Lambda+\mathcal P_M \mathfrak T_{m, M, high}^\Lambda
$$
where
$$
\mathfrak T_{m, M, low}^\Lambda:= \mathfrak T_{m}^\Lambda[\mathcal P_{\leq M} F, \mathcal P_{\leq M} G,\mathcal P_{\leq M} H]
$$
and 
$$
\mathfrak T_{m, M, high}^\Lambda:=\mathfrak T_{m}^\Lambda[\mathcal P_{\geq 2M} F, G, H]+ \mathfrak T_{m}^\Lambda[\mathcal P_{\leq M} F, \mathcal P_{\geq 2M} G, H]+\mathfrak T_{m}^\Lambda[\mathcal P_{\leq M} F, \mathcal P_{\leq M} G, \mathcal P_{\geq 2M} H].
$$

 The bound on the contribution of $\mathfrak T_{m, M, high}^\Lambda$ is straightforward: For example, for the term with $\mathcal P_{\geq 2N} F$, we have
\begin{eqnarray*}
\sum_{M \geq 1} M^{2s}\big\|\mathcal P_M \mathfrak T_{m}^\Lambda[\mathcal P_{\geq 2M} F, G, H]\big\|_{L_{x,y}^2}^2&\leq& K^2 \sum_{M \geq 1} M^{2s}\|\mathcal P_{\geq 2M} F\|_{L_{x,y}^2}^2 \|G\|_{\mathcal B}^2\|H\|_{\mathcal B}^2\\
&\lesssim& K^2 \|F\|_{H^s}^2\|G\|_{\mathcal B}^2\|H\|_{\mathcal B}^2.
\end{eqnarray*}

We now treat the contribution of the term  $\mathfrak T_{m, M, low}^\Lambda$. By~\eqref{Full Sobolev2} we can consider the localizations in $x$ and in $y$ separately. \medskip

$\bullet$ When $\mathcal P_M=Q_M$ is the localization in the $x-$direction, one can bound the contribution of $\mathfrak T_{m, M, low}^\Lambda$ as follows: 
\begin{eqnarray*}
M^s \|Q_M\mathfrak T_{m, M, low}^\Lambda\|_{L_{x,y}^2}&\lesssim& M^{-s}\big\|\partial_{x}^{2s}Q_M\mathfrak T_{m, M, low}^\Lambda\big\|_{L_{x,y}^2}\\[4pt]
&=& M^{-s}\big\|\partial_{x}^{2s}Q_M\mathfrak T_{m}^\Lambda[Q_{\leq M} F, Q_{\leq M} G,Q_{\leq M} H]\big\|_{L_{x,y}^2}\\[4pt]
&\lesssim&M^{-s}\sum_{a+b+c\leq 2s}\sum_{M_{1}, M_{2}, M_3 \leq M}\big\|Q_M\mathfrak T_{m}^\Lambda[\partial_x^{a}   Q_{M_{1}} F, \partial_x^{b} Q_{M_{2}} G, \partial_{x}^{c} Q_{M_3}H]\big\|_{L_{x,y}^2}.
\end{eqnarray*}

Assuming without loss of generality that $M_{1}\geq M_{2}, M_3$ (the other cases being similar), we can bound the above by 
\begin{align*}
\lesssim K  \sum_{M_{1} \leq M}\left(\frac{M_{1}}{M}\right)^s(M_{1}^s\|Q_{M_{1}} F\|_{L_{x,y}^2})\|G\|_{\mathcal B}\|H\|_{\mathcal B}
\end{align*}
which is square-summable in $M$ by Schur's test. \medskip

$\bullet$  The case when $\mathcal P$ is the Littlewood-Paley projection in the $y-$variable is a bit more tedious and depends on whether the operator $\mathcal K$ in~\eqref{triop} is $\mathcal J$ or $\mathcal R$. When $\mathcal J=\mathcal R$, we argue exactly as above for $\mathcal P=Q_M$ thanks to the Leibniz-type rule:
$$
\HH \mathcal R[\mathfrak f, \mathfrak g, \mathfrak h]= \mathcal R[\HH \mathfrak f, \mathfrak g, \mathfrak h]-\mathcal R[\mathfrak f, \HH \mathfrak g, \mathfrak h]+ \mathcal R[\mathfrak f, \mathfrak g, \HH \mathfrak h].
$$
On the other hand, if $\mathcal K=\mathcal J^{t}$, then using the Lemma~\ref{lem.com} below we have:
\begin{align*}
M^s \|\mathcal P_M\mathfrak T_{m, M, low}^\Lambda\|_{L_{x,y}^2}\lesssim& M^{-s}\|\mathcal H^s \mathcal P_M\mathfrak T_{m, M, low}^\Lambda\|_{L_{x,y}^2} \\
\lesssim&M^{-s}\sum_{|\alpha|+|\beta|+|\gamma|\leq 2s}\sum_{M_{1}, M_{2}, M_3 \leq M}\big\|\mathcal P_M\mathfrak T_{m}^\Lambda[ \Xi^\alpha \mathcal P_{M_{1}} F, \Xi^\beta \mathcal P_{M_{2}} G, \Xi^\gamma \mathcal P_{M_3}H]\big\|_{L_{x,y}^2}
\end{align*}
where $\Xi=(\partial_{y_{1}}, \ldots, \partial_{y_d}, y_{1},\ldots, y_d)$ and $\alpha, \beta, \gamma\in \N^{2d}$ are multi-indices.

 Assuming again without loss of generality that $M_3 \leq M_{2} \leq M_{1}$ and using the fact that $\mathcal B$ is admissible, we can bound the above sum by
$$
\lesssim K\sum_{M_{1}\leq M}\left(\frac{M_{1}}{M}\right)^s (M_{1}^s \|\mathcal P_{M_{1}} F\|_{L_{x,y}^2})\|G\|_{\mathcal B}\|H\|_{\mathcal B},
$$
 which is square-summable in $M$. \ligne
 
Proof of~\eqref{CS+}:  we first notice that $\|xF\|_{S}$ can be bounded by the RHS of~\eqref{CS+} by combining~\eqref{xtm},~\eqref{CS}, and~\eqref{FL2}.  The part of the $S^+$ norm involving $(1-\partial_{xx})^4$ can be bounded in a similar fashion as above. This concludes the proof.
\end{proof}
\ligne
Finally, we state an elementary commutation result (see also~\cite[Lemma 4.1]{AnCaSi} for more properties of these operators)
\begin{lemma}\label{lem.com}
For all $t\in \R$ and $1\leq j\leq d$
\begin{equation}\label{com1}
\partial_{y_j} e^{it\HH}=e^{it\HH}\big(\cos(2t)  \partial_{y_j}+i\sin(2t){y_j} \big), 
\end{equation}
and
\begin{equation}\label{com2}
{y_j} e^{it\HH}=e^{it\HH}\big(i\sin(2t)  \partial_{y_j}+\cos(2t){y_j} \big).
\end{equation}
\end{lemma}

\begin{proof}
For $f\in \mathcal{S}(\R^d)$, denote by $\varphi(t)=e^{-itH}\partial_{y_j}e^{itH}f$. We have $[H, y_j]=-2\partial_{y_j}$ and $[H, \partial_{y_j}]=-2{y_j}$. Then 
\begin{equation*}
\varphi'(t)=2ie^{-itH}{y_j}e^{itH}f,\quad \varphi''(t)=-4e^{-itH}\partial_{y_j}e^{itH}f=-4\varphi(t).
\end{equation*}
By solving the differential equation, we get~\eqref{com1}, and by computing $\varphi'$, we get~\eqref{com2}.
\end{proof}


\begin{thebibliography}{73}

\bibitem{AnCaSi} P. Antonelli, R. Carles and J. D. Silva, Scattering for nonlinear Schr\"{o}dinger equation under partial harmonic confinement, preprint arXiv:1310.1352.

\bibitem{BourgainQuasi} J. Bourgain, Quasi-periodic solutions of Hamiltonian perturbations of 2D linear Schr\"odinger equations, {\it Ann. of Math.,}(2) 148 (1998), no. 2, 363--439. 

\bibitem{Ca} R. Carles, Geometric optics and long range scattering for one-dimensional nonlinear Schr\"{o}dinger equations.
{\it Comm. Math. Phys.}, 220 (2001), no. 1, 41-67.

\bibitem{Carles2}
R.~Carles, Rotating points for the conformal NLS scattering operator, {\it Dynamics of PDE } 6 (2009), 35--51. 

\bibitem{CKSTTSIMA} J. Colliander,  M. Keel, G.  Staffilani, H. Takaoka and  T. Tao, Global well-posedness for Schr\"{o}dinger equations with derivative, {\it SIAM J. Math. Anal.}, 33 (2001), 649--669.

\bibitem{CSS} P. L. Christiansen, M. P. Sorensen and  A. C. Scott, {\it Nonlinear Science at the Dawn of the 21st Century}, Lecture Notes in Physics Volume 542 2000, ISBN: 978-3-540-66918-0.

\bibitem{CKSTTTor}  J. Colliander, M. Keel, G. Staffilani, H. Takaoka and T. Tao, Transfer of energy to high frequencies in the cubic defocusing nonlinear Schr\"{o}dinger equation. {\it Invent. Math.}, 181 (2010), no. 1, 39--113.

\bibitem{EliaKuk} L. H. Eliasson and S.~Kuksin, KAM for the nonlinear Schr\"{o}dinger equation, {\it Ann. of Math.}, (2) 172 (2010), no. 1, 371--435. 

\bibitem{FGH} E. Faou, P. Germain and Z. Hani, The weakly nonlinear large box limit of the 2D cubic nonlinear Schr\"{o}dinger equation, preprint, arXiv:1308.6267.

\bibitem{Frantz} D. J. Frantzeskakis, Dark solitons in atomic Bose-Einstein condensates: from theory to experiments, {\it Journal of Physics A Mathematical and Theoretical} 01/2010; 43. DOI:10.1088/1751-8113/43/21/213001. 

\bibitem{GHT1} P. Germain, Z. Hani and L. Thomann, On the continuous resonant equation for NLS.
I. Deterministic analysis, {\it in preparation}.

\bibitem{GHT2} P. Germain, Z. Hani and L. Thomann, On the continuous resonant equation for NLS.
II. Statistical study,  {\it in preparation}.

\bibitem{GeMaSh} P. Germain, N. Masmoudi and J. Shatah, Global solutions for 3D quadratic Schr\"{o}dinger equations, {\it Int. Math. Res. Not.}, (2009), 414--432.

\bibitem{GrTh1} B. Gr\'ebert and L. Thomann,  KAM for the quantum harmonic oscillator, {\it  Comm. Math. Phys.} 307, no 2, (2011), 383--427.
 
\bibitem{GuKa} M. Guardia and V. Kaloshin, Growth of Sobolev norms in the cubic defocusing nonlinear Schr\"{o}dinger equation, {\it J. Eur. Math. Soc}, to appear. 

\bibitem{Hani} Z. Hani, Long-time strong instability and unbounded orbits for some periodic nonlinear Sch\"odinger equations,  {\it Arch. Rat. Mech. Anal.}, 211 (2014), no. 3, 929--964. 

\bibitem{HaPa} Z. Hani and B. Pausader, On scattering for the quintic defocusing nonlinear Schr\"odinger equation on $\mathbb{R}\times\mathbb{T}^2$, {\it Comm. Pure and Appl. Math.}, Vol. 67 (2014), no. 9,  1466--1542.

\bibitem{HPTV} Z. Hani, B. Pausader, N. Tzvetkov and N. Visciglia, Modified scattering for the cubic Schr\"odinger equation on product spaces and applications. arXiv 1311.2275.

\bibitem{HPTVproc} Z. Hani, B. Pausader, N. Tzvetkov and N. Visciglia, Growing Sobolev norms for the cubic defocusing Schr\"odinger equation.
S\'eminaire: \'Equations aux D\'eriv\'ees Partielles. 2013--2014, Exp. No. XVI, 11 pp., S\'emin. \'Equ. D\'eriv. Partielles, \'Ecole Polytech., Palaiseau, 2014.

\bibitem{HaNa} N. Hayashi and P. Naumkin, Asymptotics for large time of solutions to the nonlinear Schr\"odinger and Hartree equations. {\it Amer. J. Math.}, 120 (1998), no. 2, 369--389.

\bibitem{Helffer}
B.~Helffer,
\newblock {Semi-classical analysis for the {S}chr\"odinger operator and
  applications}, volume 1336 of {\it Lecture Notes in Mathematics}.
\newblock Springer-Verlag, Berlin, 1988.

\bibitem{IfTa} M. Ifrim and D. Tataru, Global bounds for the cubic nonlinear Schr\"odinger equation NLS in one space dimension. {\it Preprint : arXiv:1404.7581.  }

\bibitem{Jao} C. Jao, The energy-critical quantum harmonic oscillator, {\it preprint}, arXiv:1406.2289v1 [math.AP].

\bibitem{KaPu} J. Kato and F. Pusateri, A new proof of long range scattering for critical nonlinear Schr\"{o}dinger equations, {\it J. Diff. Int. Equ.}, Vol. 24, no. 9--10 (2011).


\bibitem{KFC} P. Kevrekidis, D. Frantzeskakis and R. Carretero-Gonz\`alez, Emergent nonlinear phenomena in Bose-Einstein condensates, Springer (2008), 398pp. 

\bibitem{KV} R. Killip, M. Visan and X. Zhang, Energy-critical NLS with quadratic potential, {\it Comm. Partial Differential Equations}, 34(10-12):1531--1565, 2009.

\bibitem{KukPo} S. Kuksin and J. P\"oschel, Invariant Cantor manifolds of quasi periodic oscillations for a nonlinear Schr\"{o}dinger equation,
{\it Ann. of Math.}, (2) 143 (1996), 149--179. 

\bibitem{Lew} M. Lewin, Limite de champ moyen et condensation de Bose-Einstein, {\it Gazette des Math\'ematiciens}, (2014).      http://hal.archives-ouvertes.fr/hal-00916829.

\bibitem{LinSof} H. Lindblad and A. Soffer, Scattering and small data completeness for the critical nonlinear Schr\"odinger equation. Nonlinearity 19 (2006), no. 2, 345--353.

\bibitem{PGHH} V. P\'erez-Garc\'ia, H. Michinel and H. Herrero, Bose-Einstein solitons in highly asymmetric traps,
{\it Phys. Rev.~A}, Vol. 57, no. 5 (1998),  3837--3842.

\bibitem{Poiret}  
A. Poiret,
\newblock Solutions globales pour l'\'equation de Schr\"odinger cubique en dimension 3.\\
\newblock{\it  Preprint : arXiv:1207.1578.} 

\bibitem{PRT1} 
A. Poiret, D. Robert and L. Thomann, 
\newblock  Random weighted Sobolev inequalities on $\mathbb{R}^d$ and application to Hermite functions.  {\it  Ann. Henri Poincar\'e}, to appear (DOI: 10.1007/s00023-014-0317-5). 

\bibitem{Procesi} 
M. Procesi and C. Procesi, 
A KAM algorithm for the resonant non--linear Schr\"odinger equation, 
\newblock {\it preprint :  arXiv:1211.4242.}

\bibitem{Ramond} T. Ramond,
\newblock Course: Analyse semiclassique, {r}{\'e}sonance  et  contr{\^o}le  de l'{\'e}quation de {S}chr{\"o}dinger. 
 \url{www.math.u-psud/~ramond/docs/m2/cours.pdf} 

 \bibitem{Tao}
  T. Tao,
\newblock A pseudoconformal compactification of the nonlinear Schr\"odinger equation and applications. 
\newblock \textit{New York J. Math.} 15 (2009), 265--282. 

\bibitem{TzVi} N. Tzvetkov and N. Visciglia,
\newblock Small data scattering for the nonlinear {S}chr\"odinger equation on product spaces, 
\newblock {\it Comm. Part. Diff. Eqs.}, vol. 37, 2012, n.1, pp. 125--135.


\bibitem{YajimaZhang2}
K.~Yajima and G.~Zhang,
\newblock Local smoothing property and Strichartz inequality for Schr\"odinger equations with potentials superquadratic at infinity.
\newblock { \it J. Differential Equations} (2004), no. 1, 81--110.


\end{thebibliography}
\end{document}